\documentclass[11pt,leqno]{siamltex}
\usepackage[top=1in, bottom=1in, left=1in, right=1in]{geometry}
\usepackage{amsfonts}
\usepackage{amsmath}
\usepackage{amssymb}
\usepackage{bm}
\usepackage{subfig}
\usepackage{graphicx}
\usepackage{algorithm}
\usepackage{algorithmic}
\usepackage{multirow}
\usepackage{xcolor}

\newcommand{\Bx}{\bm{x}}
\newcommand{\Be}{\bm{e}}
\newcommand{\Bb}{\bm{b}}
\newcommand{\BW}{\bm{W}}

\newcommand{\ts}{\text{s}}

\newcommand{\tn}{\text{n}}
\newcommand{\tb}{\text{b}}
\newcommand{\tS}{\text{S}}
\newcommand{\tL}{\text{L}}
\newtheorem{remark}{Remark}[section]

\title{SelectNet: Self-Paced Learning for High-dimensional Partial Differential Equations}

\author{Yiqi Gu
\vspace{0.1in}\\
Department of Mathematics, National University of Singapore, 10 Lower Kent Ridge Road, Singapore, 119076 ({\tt matguy@nus.edu.sg})
  \vspace{0.1in}\\
Haizhao Yang\footnote{Corresponding author.}
  \vspace{0.1in}\\
  Department of Mathematics, Purdue University, West Lafayette, IN 47907,  USA ({\tt haizhao@purdue.edu})
\vspace{0.1in}\\
Chao Zhou
\vspace{0.1in}\\
Department of Mathematics, National University of Singapore, 10 Lower Kent Ridge Road, Singapore, 119076 ({\tt matzc@nus.edu.sg})
}

\begin{document}
\maketitle
\begin{abstract}
The least squares method with deep neural networks as function parametrization has been applied to solve certain high-dimensional partial differential equations (PDEs) successfully; however, its convergence is slow and might not be guaranteed even within a simple class of PDEs. To improve the convergence of the network-based least squares model, we introduce a novel self-paced learning framework, SelectNet, which quantifies the difficulty of training samples, treats samples equally in the early stage of training
, and slowly explores more challenging samples, e.g., samples with larger residual errors, mimicking the human cognitive process for more efficient learning. In particular, a selection network and the PDE solution network are trained simultaneously; the selection network adaptively weighting the training samples of the solution network achieving the goal of self-paced learning. Numerical examples indicate that the proposed SelectNet model outperforms existing models on the convergence speed and the convergence robustness, especially for low-regularity solutions.
\end{abstract}



\pagestyle{myheadings}
\thispagestyle{plain}
\markboth{SELECTNET FOR SOLVING HIGH-DIMENSIONAL PDES}{SELECTNET FOR SOLVING HIGH-DIMENSIONAL PDES}

\section{Introduction}
High-dimensional partial differential equations (PDEs) are important tools in physical, financial, and biological models \cite{Lee2002,Ehrhardt2008,Yserentant2005,Gaikwad2009,Wales2003}. However, developing numerical methods for high-dimensional PDEs has been challenging due to the curse of dimensionality in the discretization of the problem. For example, in traditional methods such as finite difference methods and finite element methods, $O(N^d)$ degree of freedom is required for a $d$-dimensional problem if we set $N$ grid points or basis functions in each direction to achieve $O(\frac{1}{N})$ accuracy. Even if $d$ becomes moderately large, the exponential growth $N^d$ in the dimension $d$ makes traditional methods immediately computationally intractable.

Recent research of the approximation theory of deep neural networks (DNNs) shows that deep network approximation is a powerful tool for mesh-free function parametrization. The research on the approximation theory of neural networks traces back to the pioneering work \cite{Cybenko1989,Hornik1989,Barron1993} on the universal approximation of shallow networks with sigmoid activation functions. The recent research focus was on the approximation rate of DNNs for various function spaces in terms of the number of network parameters showing that deep networks are more powerful than shallow networks in approximation efficiency. For example, smooth functions \cite{NIPS2017_7203,DBLP:journals/corr/LiangS16,yarotsky2017,DBLP:journals/corr/abs-1807-00297,Montanelli2019_3,suzuki2018adaptivity,Weinan2019,Weinan2019APE,E2019_2}, piecewise smooth functions \cite{PETERSEN2018296}, band-limited functions \cite{Montanelli2019}, continuous functions \cite{yarotsky18a,SHEN201974,Shen2019}. The reader is referred to \cite{Shen2019} for the explicit characterization of the approximation error for networks with an arbitrary width and depth.

In particular, deep network approximation can lessen or overcome the curse of dimensionality under certain circumstances, making it an attractive tool for solving high-dimensional problems. For functions admitting an integral representation with a one-dimensional integral kernel, {no curse of dimensionality in the approximation rate} can be shown via establishing the connection of network approximation with the Monte Carlo sampling or equivalently the law of large numbers \cite{Barron1993,Weinan2019,Weinan2019APE,E2019_2,Montanelli2019}. Based on the Kolmogorov-Arnold superposition theorem, for general continuous functions, \cite{MAIOROV199981,GUL} showed that three-layer neural networks with advanced activation functions can avoid the curse of dimensionality and the total number of parameters required is only $O(d)$; \cite{Montanelli2019_2} proves that deep ReLU network approximation can lessen the curse of dimensionality, if target functions are restricted to a space related to the constructive proof of the Kolmogorov-Arnold superposition theorem in \cite{braun}. If the approximation error is only concerned on a low-dimensional manifold, there is no curse of dimensionality for deep network approximation {in terms of the approximation error} \cite{Chui2018,Cai2018,Shen2019}. Finally, there is also extensive research showing that deep network approximation can overcome the curse of dimensionality when they are applied to approximation certain PDE solutions, e.g. \cite{HJKN19_814,hutzenthaler2020}.

As an efficient function parametrization tool, neural networks have been applied to solve PDEs via various approaches. Early work in \cite{Lee1990} applies neural networks to approximate PDE solutions defined on grid points. Later in \cite{Dissanayake1994,Lagaris1998}, DNNs are employed to approximate solutions in the whole domain, and PDEs are solved by minimizing the discrete residual error in the $L^2$-norm at prescribed collocation points. DNNs coupled with boundary governing terms by design can satisfy boundary conditions \cite{Malek2006}. Nevertheless, designing boundary governing terms is usually difficult for complex geometry. Another approach to enforcing boundary conditions is to add boundary errors to the loss function as a penalized term and minimize it as well as the PDE residual error \cite{Gobovic1994,Lagaris2000}. The second technique is in the same spirit of least squares methods in finite element methods and is more convenient in implementation. Therefore, it has been widely utilized for PDEs with complex domains. However, network computation was usually expensive, limiting the applications of network-based PDE solvers. Thanks to the development of GPU-based parallel computing over the last two decades, which greatly boosts the network computation, network-based PDE solvers were revisited recently and have become a popular tool, especially for high-dimensional problems \cite{E2017,Ritz,Han2018,Khoo2017SolvingPP,Sirignano2018,Berg2018,Zang2019,Li2019,Beck2019, Hutzenthaler2018,Hutzenthaler2019,Cai2019,RAISSI2019686,doi:10.1063/1.5110439}. Nevertheless, most network-based PDE solvers suffer from robustness issues: their convergence is slow and might not be guaranteed even within a simple class of PDEs.

To ease the issue above, we introduce a novel self-paced learning framework, SelectNet, to adaptively choose training samples in the least squares model. Self-paced learning \cite{Kumar2010} is a recently raised learning technique that can choose a part of the training samples for actual training over time. Specifically, for a training data set with $n$ samplings, self-paced learning uses a vector $v\in\{0,1\}^n$ to indicate whether each training sample should be included in the current training stage. The philosophy of self-paced learning is to simulate human beings' learning style, which tends to learn easier aspects of a learning task first and deal with more complicated samples later. Based on self-paced learning, a novel technique for selected sampling is put forward, which uses a selection neural network instead of the 0-1 selection vector $v$. Hence, it learns to avoid redundant training information and speeds up the convergence of learning outcomes. This idea is further improved in \cite{Jiang2017} by introducing a DNN to select training data for image classification. Among similar works, a state-of-the-art algorithm named SelectNet is proposed in \cite{Liu2019} for image classification, especially for imbalanced data problems. Based on the observation that samples near the singularity of the PDE solution are rare compared to samples from the regular part, we extend the SelectNet \cite{Liu2019} to network-based least squares models, especially for PDE solutions with certain irregularity. As we shall see later, numerical results show that the proposed model is competitive with the traditional (basic) least squares model for analytic solutions, and it outperforms others for low-regularity solutions, in the aspect of the convergence speed. It is worth noting that our proposed SelectNet model is essentially tuning the weights of training points to realize the adaptive sampling. Another approach is to change the distribution of training points, such as the residual-based adaptive refinement method \cite{Kaufmann2010}.

The organization of this paper is as follows. In Section 2, we introduce the least squares methods and formulate the corresponding optimization model. In Section 3, we present the SelectNet model in detail. In Section 4, we put forward the error estimates of the basic and SelectNet models. In Section 5, we discuss the network implementation in the proposed model. In Section 6, we present ample numerical experiments for various equations to validate our model. We conclude with some remarks in the final section.

\section{Least Squares Methods for PDEs}
In this work, we aim at solving the following (initial) boundary value problems, giving a bounded domain $\Omega\subset\mathbb{R}^d$:
\begin{itemize}
  \item elliptic equations
  \begin{equation}\label{01_1}
  \begin{split}
  &\mathcal{D}_x u(x)=f(x),\text{~in~}\Omega,\\
  &\mathcal{B}_x u(x)=g_0(x),\text{~on~}\partial\Omega;
  \end{split}
  \end{equation}
  \item parabolic equations
  \begin{equation}\label{01_2}
  \begin{split}
  &\frac{\partial u(x,t)}{\partial t}-\mathcal{D}_x u(x,t)=f(x,t),\text{~in~}\Omega\times(0,T),\\
  &\mathcal{B}_x u(x,t)=g_0(x,t),\text{~on~}\partial\Omega\times(0,T),\\
  &u(x,0)=h_0(x),\text{~in~}\Omega;
  \end{split}
  \end{equation}
  \item hyperbolic equations
  \begin{equation}\label{01_3}
  \begin{split}
  &\frac{\partial^2 u(x,t)}{\partial t^2}-\mathcal{D}_x u(x,t)=f(x,t),\text{~in~}\Omega\times(0,T),\\
  &\mathcal{B}_x u(x,t)=g_0(x,t),\text{~on~}\partial\Omega\times(0,T),\\
  &u(x,0)=h_0(x),\quad\frac{\partial u(x,0)}{\partial t}=h_1(x)\text{~in~}\Omega;
  \end{split}
  \end{equation}
\end{itemize}
where $u$ is the solution function; $f$, $g_0$, $h_0$, $h_1$ are given data functions; $\mathcal{D}_x$ is a spatial differential operator concerning the derivatives of $x$; $\mathcal{B}_x$ is a boundary operator specifying a Dirichlet, Neumann or Robin boundary condition.

In this method, the temporal variable $t$ will be regarded as an extra spatial coordinate, and it will not be dealt with differently from $x$. For simplicity, the PDEs in \eqref{01_1}-\eqref{01_3} are unified in the following form
\begin{equation}\label{01}
\begin{split}
&\mathcal{D}u(\Bx)=f(\Bx),\text{~in~}Q,\\
&\mathcal{B}u(\Bx)=g(\Bx),\text{~in~}\Gamma,
\end{split}
\end{equation}
where $\Bx$ includes the spatial variable $x$ and possibly the temporal variable $t$; $\mathcal{D}u=f$ represents a generic PDE; $\mathcal{B}u=g$ represents the governing conditions including the boundary condition and possibly the initial condition; $Q$ and $\Gamma$ are the corresponding domains of the equations.

Now we seek a neural network $u(\Bx;\theta)$ approximating the solution $u(\Bx)$ of the PDE \eqref{01}. Note the residual errors for the PDE and the governing conditions can be written by
\begin{equation}
\mathcal{R}_Q(u(\Bx;\theta)):=\mathcal{D}u(\Bx;\theta)-f(\Bx),\quad\mathcal{R}_\Gamma(u(\Bx;\theta)):=\mathcal{B}u(\Bx;\theta)-g(\Bx).
\end{equation}

One can solve the PDE by searching for the optimal parameters of the network that minimize these residual errors, i.e.
\begin{equation}\label{03}
\underset{\theta}{\min}~\|\mathcal{R}_Q(u(\Bx;\theta))\|_Q^2+\lambda\|\mathcal{R}_\Gamma(u(\Bx;\theta))\|_\Gamma^2,
\end{equation}
where $\|\cdot\|_\ast$ is usually the $L^2$-norm and $\lambda$ is a parameter for weighting the sum, e.g.,
\begin{equation}\label{04}
\underset{\theta}{\min}~\mathbb{E}_{\Bx\in Q}\left[|\mathcal{D}u(\Bx;\theta)-f(\Bx)|^2\right]+\lambda\mathbb{E}_{\Bx\in\Gamma}\left[|\mathcal{B}u(\Bx;\theta)-g(\Bx)|^2\right].
\end{equation}

\section{SelectNet Model}
\label{sec_sn}

The network-based least squares model has been applied to solve certain high-dimensional PDEs successfully. However, its convergence is slow and might not be guaranteed. To ease this issue, we introduce a novel self-paced learning framework, SelectNet, to adaptively choose training samples in the least squares model. The basic philosophy is to mimic the human cognitive process for more efficient learning: learning first from easier examples and slowly exploring more complicated ones. The proposed model is related to selected sampling \cite{Csiba:2018:ISM:3291125.3291152,DBLP:journals/corr/abs-1803-00942}, an important tool of deep learning for computer science applications.  Nevertheless, the effectiveness of selected sampling in scientific computing has not been fully explored yet.

In particular, a selection network $\phi_\ts(\bm{x};\theta_\ts)$ (subscript $\ts$ for ``selection") and the PDE solution network $u(\bm{x};\theta)$ are trained simultaneously; the selection network adaptively weighting the training samples of the solution network achieving the goal of self-paced learning. $\phi_\ts(\bm{x};\theta_\ts)$ is a ``mentor" helping to decide whether a sample $\bm{x}$ is important enough to train the ``student" network $u(\bm{x};\theta)$. The ``mentor" could avoid redundant training information and help to speed up the convergence. This idea is originally from self-paced learning \cite{KPK2010} and is further improved in \cite{Jiang2017} by introducing a DNN to select training data for image classification. Among similar works, a state-of-the-art algorithm named SelectNet was proposed in \cite{Liu2019} for image classification, especially for imbalanced data problem. Based on the observation that samples near the singularity of the PDE solution are rare compared to samples from the regular part, we extend the SelectNet \cite{Liu2019} to network-based least squares models, especially for PDE solutions with certain irregularity.

Originally in image classification, for a training data set $\mathcal{D} = \{(\bm{x}_i, {y}_i))\}_{i=1}^n$, self-paced learning uses a vector $\bm{v}\in \left\{0,1\right\}^n$ to indicate whether or not each training sample should be included in the current training stage ($v_i = 1$ if the $i$th sample is included in the current iteration). The overall target function including $\bm{v}$ is
\begin{equation}\label{spl}
{\min}_{\theta, \bm{v} \in \{0,1\}^n} \sum_{i=1}^n v_i\mathcal{L}(y_i, \phi(\bm{x}_i;\theta)) - \lambda \sum_{i=1}^n v_i,
\end{equation}
where $\mathcal{L}(y_i, \phi(\bm{x}_i; \theta))$ denotes the loss function of a DNN $\phi(\bm{x}_i; \theta)$ for classifying a sample $\bm{x_i}$ to $y_i$. When this model is relaxed to $\bm{v}\in [0,1]^n$ and the alternative convex search is applied to solve the relaxed optimization, a straightforward derivation easily reveals a rule for the optimal value for each entry $v_i^{(t)}$ in the $t$-th iteration as
\begin{equation}
v_i^{(t)} =   1,  \textrm{ if }\mathcal{L}(y_i, \phi(\bm{x}_i;\theta^{(t)})) < \lambda,\qquad \text{ and }\qquad v_i^{(t)} =
  0,  \textrm{ otherwise.}
\end{equation}
A sample with a smaller loss than the threshold $\lambda$ is treated as an ``easy" sample and will be selected in training. {Let us assume that the variables $\bm{v}$ and $\theta$ are trained alternatively.} When computing $\theta^{(t+1)}$ with a fixed $\bm{v}^{(t)}$, the classifier is trained only on the selected ``easy" samples. {When computing $\bm{v}^{(t+1)}$ with a fixed $\theta^{(t+1)}$, the vector $\bm{v}$ help to adjust the training samples to be used in computing $\theta^{(t+2)}$. It was shown by extensive numerical experiments that} this mechanism helps to reduce the generalization error for image classification when the training data distribution is usually different from the test data distribution \cite{KPK2010}. {In \cite{Jiang2017,Liu2019}, a selection network $\phi_s(\bm{x};\theta_s)\in[0,1]$ is trained to select training samples instead of using the binary vector $\bm{v}$ with the following loss function:
\begin{equation}\label{eqn:contopt}
{\min}_{{\theta}, {\theta}_s } \sum_{i=1}^n \phi_s(\bm{x}_i;{\theta}_s) \mathcal{L}(y_i, \phi(\bm{x}_i;{\theta})) - \lambda \sum_{i=1}^n \phi_s(\bm{x}_i;{\theta}_s).
\end{equation}
The introduction of the selection network has mainly three advantages. First, it changes the discrete optimization problem in \eqref{spl} to a continuous optimization problem in \eqref{eqn:contopt} that is much easier to solve. Besides, the selection network with values in $[0,1]$ can more adaptive adjust the weights to each sample. Finally, the number of parameters in the selection network can be much smaller than the size of $\bm{v}$, since usually a small selection network is good enough to decide weights roughly.
}

{The self-paced idea can also be applied to the preceding least squares model for solving PDEs. One naive way is to rewrite the optimization \eqref{04} as
\begin{equation}\label{02}
\underset{\theta}{\min}~ \frac{1}{N_1}\sum_{i=1}^{N_1}v'_i|\mathcal{D}u(\Bx^1_i;\theta)-f(\Bx^1_i)|^2+\frac{\lambda}{N_2}\sum_{i=1}^{N_2}v''_j|\mathcal{B}u(\Bx^2_i;\theta)-g(\Bx^2_i)|^2,
\end{equation}
where $\{\Bx^1_i\}_{i=1}^{N_1}\subset\Omega$ and $\{\Bx^2_i\}_{i=1}^{N_2}\subset\partial\Omega$ are random samples; $v'_i$ and $v''_i$ are adaptive binary weights denoting if the samples are selected or not in the loss. Similar adaptive sampling techniques can be found in \cite{Nakamura2019,E2020}. Solving PDEs using deep learning is different from conventional supervised learning, where sample data are fixed without the flexibility to be arbitrary in the problem domain. The training and testing data distributions are the same, and there is no limitation for sampling when we solve PDEs. Therefore, appropriately selecting training data and assigning weights $\bm{v}'$ and $\bm{v}''$ in each optimization iteration can better facilitate the convergence of deep learning to the true PDE solution.

Intuitively, a good strategy is to first choose ``easy" samples to quickly identify a rough PDE solution and then use more ``difficult" samples with large residual errors to refine the PDE solution. For example, in the early stage of the training, random samples are uniformly drawn in the PDE domain; in the latter stage of the training, we can select samples with almost the highest residual errors for training. However, this naive selection strategy might be too greedy: large residual errors usually occur where the PDE solution is irregular (e.g., near low regularity points), resulting in selected training samples gathering around these ``difficult" points with few samples in other regions. Note that deep neural networks are functions globally supported in the PDE domain. Training with samples restricted in a small area may lead to large test errors in other areas. In our experiments, we observe that this naive selection strategy applied to \eqref{02} even works worse than the basic model \eqref{04} (See the numerical example in Section \ref{Sec_compare_binary}).

Borrowing the idea in \cite{Jiang2017,Liu2019}, we introduce two neural networks, $\phi_\ts'(\Bx;\theta_\ts')$ and $\phi_\ts''(\Bx;\theta_\ts'')$, named as the selection network for the PDE residual error and the boundary condition error, respectively, to replace $\bm{v}'$ and $\bm{v}''$ in \eqref{02}. The introduction of selection networks admits three main advantages over the naive binary weights, as discussed previously for the models in \eqref{spl} and \eqref{eqn:contopt}. According to the discussion in the last paragraph, the selection networks $\phi_\ts'(\Bx;\theta_\ts')$ and $\phi_\ts''(\Bx;\theta_\ts'')$ should satisfy the following requirements. 1) As weight functions, they are required to be non-negative and bounded. 2) They should not have a strong bias for weighting samples in the early stage of training. 3) They prefer higher weights for samples with larger point-wise residual errors in the latter stage of training.

For the first requirement, $\phi_\ts'(\Bx;\theta_\ts')$ and $\phi_\ts''(\Bx;\theta_\ts'')$ are enforced to satisfy
\begin{gather}
m_0<\phi_\ts'(\Bx;\theta_\ts')<M_0,\quad\forall\Bx\in Q\text{~and~}\forall\theta_\ts',\label{06_1}\\
m_0<\phi_\ts''(\Bx;\theta_\ts'')<M_0,\quad\forall\Bx\in \Gamma\text{~and~}\forall\theta_\ts'',\label{06_2}
\end{gather}
where $M_0>1>m_0\geq0$ are prescribed constants. Note the conditions \eqref{06_1}-\eqref{06_2} hold automatically if the last layer of activation functions of $\phi_\ts'(\Bx;\theta_\ts')$ and $\phi_\ts''(\Bx;\theta_\ts'')$ is bounded (e.g., using a tanh or sigmoid activation function) and the network output is properly re-scaled and shifted as we shall discuss later in the next section. Therefore, the corresponding weighted least squares method is formulated by
\begin{equation}\label{11}
\mathbb{E}_{x\in Q}\left[\phi_\ts'(\Bx;\theta_\ts')|\mathcal{D}u(\Bx;\theta)-f(\Bx)|^2\right]\\
+\lambda\mathbb{E}_{x\in\Gamma}\left[\phi_\ts''(\Bx;\theta_\ts'')|\mathcal{B}u(\Bx;\theta)-g(\Bx)|^2\right].
\end{equation}

For the second requirement, when the selection networks are randomly initialized with zero bias and random weights with a zero mean and a small variance, the selection networks are random functions close to a constant. Therefore, the selection networks have no bias in weighting samples in the early stage of training.

The third requirement can also be satisfied. Based on the principle that higher weights should be added to samples with larger point-wise residual errors, we can train $\phi_\ts'(\Bx;\theta_\ts')$ and $\phi_\ts''(\Bx;\theta_\ts'')$ via
\begin{equation}\label{16}
\underset{\theta_\ts',\theta_\ts''}{\max}~\mathbb{E}_{x\in Q}\left[\phi_\ts'(\Bx;\theta_\ts')|\mathcal{D}u(\Bx;\theta)-f(\Bx)|^2\right]\\
+\lambda\mathbb{E}_{x\in\Gamma}\left[\phi_\ts''(\Bx;\theta_\ts'')|\mathcal{B}u(\Bx;\theta)-g(\Bx)|^2\right]
\end{equation}
subject to the normalization conditions,
\begin{equation}\label{05}
\frac{1}{|Q|}\int_Q\phi_\ts'(\Bx;\theta_\ts')d\Bx=1,\quad\frac{1}{|\Gamma|}\int_\Gamma\phi_\ts''(\Bx;\theta_\ts'')d\Bx=1.
\end{equation}
Note in \eqref{16}, to achieve the maximum of the loss function, $\phi_\ts'(\Bx;\theta_\ts')$ tends to take larger values where $|\mathcal{D}u(\Bx;\theta)-f(\Bx)|$ is larger, and take smaller values elsewhere. Also, $\phi_\ts'(\Bx;\theta_\ts')$ will not take large values everywhere since it is normalized by \eqref{05}. The same mechanism is also true for $\phi_\ts''(\Bx;\theta_\ts'')$. In the latter stage of training, $\phi_\ts'(\Bx;\theta_\ts')$ and $\phi_\ts''(\Bx;\theta_\ts'')$ have been optimized by the maximization problem above to choose ``difficult" samples and, hence, the third requirement above is satisfied.

For simplicity, we can combine \eqref{16} and \eqref{05} as the following penalized optimization
\begin{multline}\label{17}
\underset{\theta_\ts',\theta_\ts''}{\max}~\mathbb{E}_{x\in Q}\left[\phi_\ts'(\Bx;\theta_\ts')|\mathcal{D}u(\Bx;\theta)-f(\Bx)|^2\right]
+\lambda\mathbb{E}_{x\in\Gamma}\left[\phi_\ts''(\Bx;\theta_\ts'')|\mathcal{B}u(\Bx;\theta)-g(\Bx)|^2\right]\\
-\varepsilon^{-1}\left[\left(\frac{1}{|Q|}\int_Q\phi_\ts'(\Bx;\theta_\ts')d\Bx -1\right)^2+\left(\frac{1}{|\Gamma|}\int_{\Gamma}\phi_\ts''(\Bx;\theta_\ts'')d\Bx-1\right)^2\right],
\end{multline}
where $\varepsilon>0$ is a small penalty constant. When $\phi_\ts'(\Bx;\theta_\ts')$ and $\phi_\ts''(\Bx;\theta_\ts'')$ are fixed, we can train the solution network $u(\Bx;\theta)$ by minimizing \eqref{17}, i.e.,
\begin{multline}\label{07}
\underset{\theta}{\min}~\underset{\theta_\ts',\theta_\ts''}{\max}~\mathbb{E}_{x\in Q}\left[\phi_\ts'(\Bx;\theta_\ts')|\mathcal{D}u(\Bx;\theta)-f(\Bx)|^2\right]\\
+\lambda\mathbb{E}_{x\in\Gamma}\left[\phi_\ts''(\Bx;\theta_\ts'')|\mathcal{B}u(\Bx;\theta)-g(\Bx)|^2\right]\\
-\varepsilon^{-1}\left[\left(\frac{1}{|Q|}\int_Q\phi_\ts'(\Bx;\theta_\ts')d\Bx -1\right)^2+\left(\frac{1}{|\Gamma|}\int_{\Gamma}\phi_\ts''(\Bx;\theta_\ts'')d\Bx-1\right)^2\right],
\end{multline}
which is the final model in the SelectNet method.

\begin{remark}
An alternate way to penalize the selection networks is to divide the residual terms in \eqref{16} by the norms of the selection networks. Namely, we solve
\begin{multline}\label{37}
\underset{\theta}{\min}~\underset{\theta_\ts',\theta_\ts''}{\max}~\|\phi_\ts'\|_{\Omega}^{-1}\mathbb{E}_{x\in Q}\left[\phi_\ts'(\Bx;\theta_\ts')|\mathcal{D}u(\Bx;\theta)-f(\Bx)|^2\right]\\
+\lambda\|\phi_\ts''\|_{\Gamma}^{-1}\mathbb{E}_{x\in\Gamma}\left[\phi_\ts''(\Bx;\theta_\ts'')|\mathcal{B}u(\Bx;\theta)-g(\Bx)|^2\right].
\end{multline}
However, in practice, the results of \eqref{37} are sensitive to the types of norms and hyperparameters; hence \eqref{37} is more challenging to obtain good numerical results than the formulation \eqref{07}.
\end{remark}

Although the introduction of SelectNet is motivated by self-paced learning in image classification, surprisingly, SelectNet can also be understood via conventional mathematical analysis. The square root of the non-negative selection networks can also be understood as the test function in the weak form of conventional PDE solvers. In the SelectNet, we apply the idea of test functions to both the PDE and the boundary condition, e.g., hoping to identify $u(\Bx;\theta)$ ensuring the following two equalities for all non-negative test functions:
\[
\left(\sqrt{\phi_\ts'(\Bx;\theta_\ts')},\mathcal{D}u(\Bx;\theta)\right)_Q=\left(\sqrt{\phi_\ts'(\Bx;\theta_\ts')}, f(\Bx)\right)_Q
\]
with $\left(\cdot,\cdot\right)_Q$ as the inner product of $L^2(Q)$ and
\[
\left(\sqrt{\phi_\ts''(\Bx;\theta_\ts'')},\mathcal{B}u(\Bx;\theta)\right)_\Gamma=\left(\sqrt{\phi_\ts''(\Bx;\theta_\ts'')}, g(\Bx)\right)_\Gamma
\]
with $\left(\cdot,\cdot\right)_\Gamma$ as the inner product of $L^2(\Gamma)$. Conventional methods apply test functions for the PDE only and the test functions are not necessarily non-negative. In the SelectNet, the integration by part is not applied so as to let the test function play a role of weighting, while conventional methods use the integration by part to weaken the regularity requirement of the PDE solution. Only a single test function is used in SelectNet with a maximum requirement to guarantee that the solution of the min-max problem is the solution of the original problem (see Theorem \ref{thm} later), while conventional methods use sufficiently many test functions that can form a set of basis functions in the discrete test function space. The idea of using test functions in deep learning was also used in \cite{Zang2019}, where the test function was used in a weak form with integration by part. The idea of using a min-max optimization problem instead of the minimization problem to solve PDEs has been studied for many decades, e.g. \cite{Friedrichs}. Maximizing over all possible test functions can obtain the best test function that amplifies the residual error the most, which can better help the minimization problem to identify the PDE solution. {When an optimization algorithm is applied to solve the min-max problem, the optimization dynamic consists of a solution dynamic that converges to the PDE solution and a test dynamic that provide a sequence of test functions to characterize the error of the numerical solution at each iteration. The training dynamic of the selection network in SelectNet approximates the test function dynamic, and the training dynamic of the solution network in SelectNet approximate the solution dynamic.  }
}

\section{Error estimates}
In this section, theoretical analysis are presented to show the solution errors of the basic and SelectNet models are bounded by the loss function (mean square of the residual). Specifically, we will take the elliptic PDE with Neumann boundary condition as an example. The conclusion can be generalized for other well-posed PDEs by similar argument. Consider
\begin{equation}\label{21}
\begin{cases}
-\Delta u+cu=f,\text{~in~}\Omega,\\
\frac{\partial u}{\partial\tn}=g,\text{~on~}\partial\Omega,
\end{cases}
\end{equation}
where $\Omega$ is an open subset of $\mathbb{R}^d$ whose boundary $\partial\Omega$ is $C^1$ smooth; $f\in L^2(\Omega)$, $g\in L^2(\partial\Omega)$, $c(x)\geq\sigma>0$ is a given function in $L^2(\Omega)$.

\begin{theorem}\label{thm}
Suppose the problem \eqref{21} admits a unique solution $u_*$ in $C^1(\overline{\Omega})$. Also, suppose the variational optimization problem
\begin{equation}\label{22}
\underset{u\in\mathcal{N}}{\min}~J(u):=\underset{u\in\mathcal{N}}{\min}\int_\Omega|-\Delta u+cu-f|^2dx+\lambda\int_{\partial\Omega}|\frac{\partial u}{\partial\tn}-g|^2dx,
\end{equation}
has an admissible set $\mathcal{N}\subset C^2(\overline{\Omega})$ containing a feasible solution $u_\tb\in\mathcal{N}$ satisfying
\begin{equation}\label{23}
J(u_\tb)<\delta,
\end{equation}
then
\begin{equation}\label{31}
\|u_\tb-u_*\|_{H^1(\Omega)}\leq c\max(1,\sigma^{-1})\max(1,\lambda^{-\frac{1}{2}})\delta^{\frac{1}{2}},
\end{equation}
where $c>0$ is a constant only depending on $d$ and $\Omega$. Furthermore, let $\mathcal{S}'$ be a subset of $\{\phi\in C(\Omega):\phi>0\}$ which contains $\phi(x)\equiv1$ for all $x\in\Omega$; let $\mathcal{S}''$ be a subset of $\{\phi\in C(\partial\Omega):\phi>0\}$ which contains $\phi(x)\equiv1$ for all $x\in\partial\Omega$. Suppose the variational optimization problem
\begin{multline}\label{24}
\underset{u\in\mathcal{N}}{\min}~J_{\mathcal{S}',\mathcal{S}''}(u):=\underset{u\in\mathcal{N}}{\min}\underset{\phi'\in\mathcal{S}',\phi''\in\mathcal{S}''}{\max}\int_\Omega\phi'|-\Delta u+cu-f|^2dx+\lambda\int_{\partial\Omega}\phi''|\frac{\partial u}{\partial\tn}-g|^2dx\\
-\varepsilon^{-1}\left[\left(\frac{1}{|\Omega|}\int_\Omega\phi'dx -1\right)^2+\left(\frac{1}{|\partial\Omega|}\int_{\partial\Omega}\phi''dx-1\right)^2\right],
\end{multline}
has a feasible solution $u_\ts\in\mathcal{N}$ satisfying
\begin{equation}\label{25}
J_{\mathcal{S}',\mathcal{S}''}(u_\ts)<\delta,
\end{equation}
then
\begin{equation}\label{33}
\|u_\ts-u_*\|_{H^1(\Omega)}\leq c\max(1,\sigma^{-1})\max(1,\lambda^{-\frac{1}{2}})\delta^{\frac{1}{2}}.
\end{equation}
\end{theorem}

\begin{proof}
Let $v_\tb:=u_\tb-u_*$. Starting from the identity
\begin{equation}\label{26}
-\Delta v_\tb+cv_\tb=-\Delta u_\tb+cu_\tb-f,
\end{equation}
we multiply $v_\tb$ to both sides of \eqref{26} and integrate over $\Omega$. Since $v_\tb\in C^1(\overline{\Omega})$, by integration by parts it follows
\begin{equation}\label{27}
\|\nabla v_\tb\|_{L^2(\Omega)}^2+\sigma\|v_\tb\|_{L^2(\Omega)}^2
\leq\int_\Omega(-\Delta u_\tb+cu_\tb-f)v_\tb dx+\int_{\partial\Omega}v_\tb\frac{\partial v_\tb}{\partial\tn}dx.
\end{equation}
Hence, by the Cauchy-Schwarz inequality,
\begin{multline}\label{28}
\min(1,\sigma)\|v_\tb\|_{H^1(\Omega)}^2\leq\|-\Delta u_\tb+cu_\tb-f\|_{L^2(\Omega)}\cdot\|v_\tb\|_{L^2(\Omega)}\\
+\|v_\tb\|_{L^2(\partial\Omega)}\cdot\|\frac{\partial u_\tb}{\partial\tn}-g\|_{L^2(\partial\Omega)}.
\end{multline}
By the trace theorem, $\|v_\tb\|_{L^2(\partial\Omega)}\leq c'\|v_\tb\|_{H^1(\Omega)}$ for some $c'>0$ only depending on $d$ and $\Omega$. Then we have
\begin{multline}\label{29}
\min(1,\sigma)\|v_\tb\|_{H^1(\Omega)}^2\\
\leq\|v_\tb\|_{H^1(\Omega)}\left(\|-\Delta u_\tb+cu_\tb-f\|_{L^2(\Omega)}+c'\|\frac{\partial u_\tb}{\partial\tn}-g\|_{L^2(\partial\Omega)}\right)\\
\leq c''\|v_\tb\|_{H^1(\Omega)}\left(\|-\Delta u_\tb+cu_\tb-f\|_{L^2(\Omega)}^2+\|\frac{\partial u_\tb}{\partial\tn}-g\|_{L^2(\partial\Omega)}^2\right)^\frac{1}{2},
\end{multline}
with $c''=\sqrt{2}\max(1,c')$. Finally, by the hypothesis \eqref{23}, \eqref{31} directly follows from \eqref{29}.

Moreover, by taking $\phi'\equiv1$, $\phi''\equiv1$ we directly have
\begin{equation}
\int_\Omega|-\Delta u+cu-f|^2dx+\lambda\int_{\partial\Omega}|\frac{\partial u}{\partial\tn}-g|^2dx\leq J_{\mathcal{S}',\mathcal{S}''}(u_\ts)<\delta.
\end{equation}
The same estimate for $\|u_\ts-u_*\|_{H^1(\Omega)}$ can be obtained by similar argument.
\end{proof}

By using the triangle inequality, we can conclude the solutions of the basic and SelectNet models are equivalent as long as the loss functions are minimized sufficiently. As an immediate result, we have the following corollary.
\begin{corollary}
Under the hypothesis of Theorem 4.1, we have
\begin{equation}\label{32}
\|u_\tb-u_\ts\|_{H^1(\Omega)}\leq c\max(1,\sigma^{-1})\max(1,\lambda^{-\frac{1}{2}})\delta^{\frac{1}{2}}.
\end{equation}
\end{corollary}

\section{Network Implementation}
\subsection{Network Architecture}\label{Sec_architecture}
The proposed framework is independent of the choice of DNNs. Advanced network design may improve the accuracy and convergence of the proposed framework, which would be interesting for future work.

In this paper, feedforward neural networks will be repeatedly applied. Let $\phi(\Bx;\theta)$ denote such a network with an input $\Bx$ and parameters $\theta$, then it is defined recursively as follows:
\begin{equation}\label{10}
\begin{split}
&\Bx^0 = \Bx,\\
&\Bx^{l+1} = \sigma(\BW^l\Bx^l+\Bb^l),\quad l=0,1,\cdots,L-1,\\
&\phi (\Bx;\theta) = \BW^L\Bx^L+\Bb^L,
\end{split}
\end{equation}
where $\sigma$ is an application-dependent nonlinear activation function, and $\theta$ consists of all the weights and biases $\{\BW^l,\Bb^l\}_{l=0}^L$ satisfying
\begin{equation}
\begin{split}
&\BW^0\in\mathbb{R}^{m\times d},\quad\BW^L\in\mathbb{R}^{1\times m},\quad\Bb^L\in\mathbb{R},\\
&\BW^l\in\mathbb{R}^{m\times m},\quad\text{for~}l=1,\cdots,L-1,\\
&\Bb^l\in\mathbb{R}^{m\times 1},\quad\text{for~}l=0,\cdots,L-1.\\
\end{split}
\end{equation}
The number $m$ is called the width of the network and $L$ is called the depth.

For simplicity, we deploy the feedforward neural network with the activation function $\sigma(\Bx)=\sin(\Bx)$ as the solution network that approximates the solution of the PDE. As for the selection network introduced in Section \ref{sec_sn}, since it is required to be bounded in $[m_0,M_0]$, it can be defined via
\begin{equation}
\phi_\ts(\Bx;\theta) = (M_0-m_0)\sigma_\ts(\hat{\phi}(\Bx;\theta))+m_0,
\end{equation}
where $\sigma_\ts(x)=1/(1+\exp(-x))$ is the sigmoidal function, and $\hat{\phi}$ is a generic network, e.g. a feedforward neural network with the ReLU activation $\sigma(\Bx)=\max\{0,\Bx\}$.

\subsection{Special Network for Dirichlet Boundary Conditions}
\label{sec_DB}

In the case of homogeneous Dirichlet boundary conditions, it is worth mentioning a special network design that satisfies the boundary condition automatically as discussed in \cite{Lagaris1998,Berg2018}.

Let us focus on the boundary value problem to introduce this special network structure. It is straightforward to generalize this idea to the case of an initial boundary value problem and we omit this discussion. Assume a homogeneous Dirichlet boundary condition
\begin{equation}
u(\Bx)=0,\quad\text{on}~\partial\Omega,
\end{equation}
then a solution network automatically satisfying the condition above can be constructed by
\begin{equation}
u(\Bx;\theta) = h(\Bx)\hat{u}(\Bx;\theta),
\end{equation}
where $\hat{u}$ is a generic network as in \eqref{10}, and $h$ is a specifically chosen function such as $h=0$ on $\Gamma$.

For example, if $\Omega$ is a $d$-dimensional unit ball, then $u(\Bx;\theta)$ can take the form
\begin{equation}\label{13}
u(\Bx;\theta) = (|\Bx|^2-1)\hat{u}(\Bx;\theta).
\end{equation}
For another example, if $\Omega$ is the $d$-dimensional cube $[-1,1]^d$, then $u(x;\theta)$ can take the form
\begin{equation}
u(\Bx;\theta) = \underset{i=1}{\overset{d}{\prod}}(x_i^2-1)\hat{u}(\Bx;\theta).
\end{equation}
Since the boundary condition $\mathcal{B}u=0$ is always fulfilled, it suffices to solve the min-max problem
\begin{equation}\label{12}
\underset{\theta}{\min}~\underset{\theta_\ts'}{\max}~\mathbb{E}_{x\in Q}\left[\phi_\ts'(\Bx;\theta_\ts')|\mathcal{D}u(\Bx;\theta)-f(\Bx)|^2\right]
-\varepsilon^{-1}\left(\frac{1}{|Q|}\int_Q\phi_\ts'(\Bx;\theta_\ts')d\Bx-1\right)^2
\end{equation}
to identify the best solution network $u(\Bx;\theta)$.

\subsection{Derivatives of Networks}
Note that the evaluation of the optimization problem in \eqref{07} involves the derivative of the network $u(\Bx;\theta)$ in terms of $\Bx$. When the activation function of the network is differentiable,  the network is differentiable and the derivative in terms of $\Bx$ can be evaluated efficiently via the back-propagation algorithm. Note that the network we adopt in this paper is not differentiable. Hence, finite difference method will be utilized to estimate the derivative of networks. For example, for the elliptic operator $\mathcal{D}u:=\nabla\cdot(a(x)\nabla u)$, $\mathcal{D}u(\Bx;\theta)$ can be estimated by the second-order central difference formula
\begin{multline}
\mathcal{D}u(\Bx;\theta)\approx\frac{1}{h^2}\overset{d}{\underset{i=1}{\sum}}a(\Bx+\frac{1}{2}h\Be_i)(u(\Bx+h\Be_i,\theta)-u(\Bx;\theta))\\
-a(\Bx-\frac{1}{2}h\Be_i)(u(\Bx;\theta)-u(\Bx-h\Be_i,\theta)),
\end{multline}
up to an error of $O(dh^2)$. In the experiments (Section \ref{Sec_experiment}), we take $h=10^{-4}$ for all examples with $d$ up to 100. Hence the truncation errors are up to $O(10^{-6})$, which are overwhelmed by the final errors (at least $O(10^{-4})$). This implies the truncation errors from finite difference can be ignored in practice.

Indeed, one can also use the automatic differentiation in TensorFlow or Pytorch based on the explicit formula of networks to evaluate the derivatives in the practical implementation, which brings no truncation errors. However, the computational cost of this approach is high when a second order (or higher) derivative is computed. Hence we choose finite difference method for derivative computation in this paper.

\subsection{Network Training}
Once networks have been set up, the rest is to train the networks to solve the min-max problem in \eqref{07}. The stochastic gradient descent (SGD) method or its variants (e.g., Adam \cite{KB2014}) is an efficient tool to solve this problem numerically. Although the convergence of SGD for the min-max problem is still an active research topic \cite{minmax1,minmax2,minmax3}, empirical success shows that SGD can provide a good approximate solution.

Before completing the algorithm description of SelectNet, let us introduce the key setup of SGD and summarize it in Algorithm \ref{Alg} below. In each training iteration, we first set uniformly distributed training points $\{\Bx^1_i\}_{i=1}^{N_1}\subset Q$  and $\{\Bx^2_i\}_{i=1}^{N_2}\subset\Gamma$ , and define the empirical loss of these training points as
\begin{multline}
J(\theta,\theta_\ts)=\frac{1}{N_1}\overset{N_1}{\underset{i=1}{\sum}}\phi_\ts'(\Bx^1_i;\theta_\ts')|\mathcal{D}u(\Bx^1_i,\theta)-f(\Bx^1_i)|^2\\
+\frac{\lambda}{N_2}\overset{N_2}{\underset{i=1}{\sum}}\phi_\ts''(\Bx^2_i;\theta_\ts'')|\mathcal{B}u(\Bx^2_i,\theta)-g(\Bx^2_i)|^2\\
-\varepsilon^{-1}\left[\left(\frac{1}{N_1}\overset{N_1}{\underset{i=1}{\sum}}\phi_\ts'(\Bx^1_i;\theta_\ts')-1\right)^2+\left(\frac{1}{N_2}\overset{N_2}{\underset{i=1}{\sum}}\phi_\ts''(\Bx^2_i;\theta_\ts'')-1\right)^2\right],
\end{multline}
where $\theta_\ts:=[\theta_\ts',\theta_\ts'']$. Next, $\theta_\ts$ can be updated by the gradient ascent via
\begin{equation}\label{08}
\theta_\ts\leftarrow\theta_\ts+\tau_\ts\nabla_{\theta_\ts} J,
\end{equation}
and $\theta$ can be updated by the gradient descent via
\begin{equation}\label{09}
\theta\leftarrow\theta-\tau\nabla_{\theta} J,
\end{equation}
with step sizes $\tau_\ts$ and $\tau$. Note that training points are randomly renewed in each iteration. In fact, for the same set of training points in each iteration, the updates \eqref{08} and \eqref{09} can be performed $n_1$ and $n_2$ times, respectively.

\begin{algorithm}
\caption{The Least Squares Model with SelectNet}
\label{Alg}
\begin{algorithmic}
\REQUIRE the PDE \eqref{01}
\ENSURE the parameters $\theta$ in the solution network $u(\Bx;\theta)$
\STATE Set parameters $n$, $n_1$, $n_2$ for iterations and parameters $N_1$, $N_2$ for sample sizes
\STATE Initialize $u(\Bx;\theta^{0,0})$ and $\phi_\ts(\Bx;\theta_\ts^{0,0})$
\FOR{$k=1,\cdots,n$}
\STATE Generate uniformly distributed sampling points\\ $\{\Bx^1_i\}_{i=1}^{N_1}\subset Q$  and $\{\Bx^2_i\}_{i=1}^{N_2}\subset\Gamma$
\FOR{$j=1,\cdots,n_1$}
\STATE Update $\theta_\ts^{k-1,j}\leftarrow\theta_\ts^{k-1,j-1}+\tau_\ts^{(k)}\nabla_{\theta_\ts} J(\theta_\ts^{k-1,j-1},\theta^{k-1,0})$
\ENDFOR
\STATE $\theta_\ts^{k,0}\leftarrow\theta_\ts^{k-1,n_1}$
\FOR{$j=1,\cdots,n_2$}
\STATE Update $\theta^{k-1,j}\leftarrow\theta^{k-1,j-1}-\tau^{(k)}\nabla_{\theta} J(\theta_\ts^{k,0},\theta^{k-1,j-1})$
\ENDFOR
\STATE $\theta^{k,0}\leftarrow\theta^{k-1,n_2}$
\IF{Stopping criteria is satisfied}
\STATE Return $\theta=\theta^{k,0}$
\ENDIF
\ENDFOR
\end{algorithmic}
\end{algorithm}

\section{Numerical Experiments}\label{Sec_experiment}
In this section, the proposed SelectNet model is tested on several PDE examples, including elliptic/parabolic and linear/nonlinear high-dimensional problems. Other network-based methods are also implemented for comparison. For all methods, we choose the feedforward architecture with activation $\sigma(x)=\max(x^3,0)$ for the solution network. Additionally, for SelectNet, we choose feedforward architecture with ReLU activation for the selection network. AdamGrad \cite{Duchi2011} is employed to solve the optimization problems, with learning rates
\begin{equation}
\tau_\ts^{(k)}=10^{-4},
\end{equation}
for the selection network, and
\begin{equation}\label{35}
\tau^{(k)}=10^{-3-3j/1000},\text{~if~}n^{(j)}<k\leq n^{(j+1)},\quad\forall j=0,\cdots,1000,
\end{equation}
for the solution network, where $0=n^{(0)}<\cdots<n^{(1000)}=n$ are equidistant segments of total iterations. Other parameters used in the model and algorithm are listed in Table \ref{Tab_parameters}. Unless otherwise specified, in all examples, we set $N_1=10000$, $N_2=10000$, $n=20000$, $n_1=1$, $\lambda=1$, $m=100$, $L=3$ for all methods and set $n_2=1$, $\varepsilon=0.001$, $m_\ts=20$, $L_\ts=3$, $m_0=0.8$, $M_0=5$ especially for SelectNet.
\begin{table}
\centering
\begin{tabular}{|c|l|}
  \hline
  $d$ & the dimension of the problem \\  \hline
  $m$ & the width of each layer in the solution network\\\hline
  $m_\ts$ & the width of each layer in the selection network \\\hline
  $L$ & the depth of the solution network \\\hline
  $L_\ts$ & the depth of the selection network \\\hline
  $M_0$ & the upper bound of the selection network \\\hline
  $m_0$ & the lower bound of the selection network \\\hline
  $n$ & number of iterations in the optimization  \\\hline
  $n_1$ & number of updates of the selection network in each iteration \\\hline
  $n_2$ & number of updates of the solution network in each iteration \\\hline
  $N_1$ & number of training points inside the domain in each iteration \\\hline
  $N_2$ & number of training points on the domain boundary in each iteration  \\\hline
  $\varepsilon$ & penalty parameter to uniform the selection network  \\\hline
  $\lambda$ & summation weight of the boundary least squares \\
  \hline
\end{tabular}
\caption{\em Parameters in the model and algorithm.}
\label{Tab_parameters}
\end{table}

We take the (relative) $\ell^2$ error at uniformly distributed testing points $\{\Bx_i\}\subset \tilde{Q}$ as the metric to evaluate the accuracy, which is formulated by
\begin{equation}
e_{\ell^2}(\theta):=\left(\frac{\underset{i}{\sum}|u(\Bx_i;\theta)-u(\Bx_i)|^2}{\underset{i}{\sum}|u(\Bx_i)|^2}\right)^\frac{1}{2}.
\end{equation}
Here $\tilde{Q}\subset Q$ is the domain for error evaluation. In all examples, we choose $10000$ testing points for error evaluation.

\subsection{Comparative Experiment}
In the first experiment, we compare the proposed SelectNet model with other network-based methods on the following 2-D Poisson equation,
\begin{equation}\label{34}
\begin{split}
-\Delta u&=1,\quad\text{in~}\Omega:=(-1,1)\times(-1,1),\\
u&=0,\quad\text{on}~\partial\Omega,
\end{split}
\end{equation}
with a solution expressed by the series
\begin{equation}
u(x_1,x_2)=-\frac{64}{\pi^4}\overset{\infty}{\underset{n,m=1 \atop n,m \text{~odd}
}{\sum}}(-1)^{\frac{n+m}{2}}\frac{cos(\frac{n\pi x_1}{2})cos(\frac{m\pi x_2}{2})}{nm(n^2+m^2)}.
\end{equation}
As a classic testing example for PDE methods, the problem \eqref{34} is well-known for the low-regularity of its solution at the four corners of $\Omega$. In this experiment, both the interior training points and testing points are chosen uniformly in the domain, and the boundary training points are chosen uniformly on the boundary. Since the numerical results are influenced by the randomness of the network initialization and the stochastic training process, we implement each method for 50 times with different seeds and compute the mean and standard deviation of the final errors.

\subsubsection{Comparison with Recent Methods}
We implement the basic least squares model, SelectNet model, and recently raised methods: deep Ritz method (DRM) \cite{Ritz} and weak adversarial networks (WAN) \cite{Zang2019} under the same setting, and compare their convergence speed. All methods are implemented for 600 seconds, with learning rates given in \eqref{35} for the first 10000 iterations and $10^{-6}$ for the subsequent iterations. The means and standard deviations of the final $\ell^2$ errors of 50 trials are listed in Table \ref{Tab_case1_errors}. For each method. We select 10 of the 50 trials to present their error curves with respect to the computing time in Figure \ref{Fig_case1_comparison_errors}. It is observed in the first 50 seconds SelectNet has the fastest error decay, and in the end, SelectNet obtains the smallest errors. We also note that for each method, the error deviations are much smaller than the error means, showing the numerical stability with respect to the stochasticity of algorithms.

Across different trials, the selection networks of the SelectNet model evolve in a nearly identical manner. From all trials, we take one to show the surfaces of the selection network at the initial stage and the 2000th, 5000th, 10000th iterations (see Figure \ref{Fig_case1_mesh}). We can clearly find that high peaks appear at the four corners over time where the solution is less regular, while other region preserves to be low and constant. This distribution will improve the convergence at the corners that are ``difficult" to deal with.

\begin{table}
\centering
\begin{tabular}{|c|c|c|c|c|}
  \hline
   & Basic & SelectNet & DRM & WAN \\\hline
    Mean of Errors $\mu$ & $7.588\times10^{-3}$ & $3.288\times10^{-4}$ & $8.681\times10^{-4}$ & $2.177\times10^{-3}$\\\hline
    Standard Deviation $\sigma$ & $1.080\times10^{-3}$ & $7.821\times10^{-5}$ & $1.072\times10^{-4}$ & $8.002\times10^{-4}$\\\hline
    Coefficient of Variation $\sigma/\mu$ & 14.2\% & 23.8\% & 12.4\% & 36.8\% \\\hline
\end{tabular}
\caption{\em Means and standard deviations of the $\ell^2$ errors obtained within 600 seconds by various methods in the comparative example. (totally 50 trials for each method)}
\label{Tab_case1_errors}
\end{table}

\begin{figure}
\centering
\includegraphics[scale=0.5]{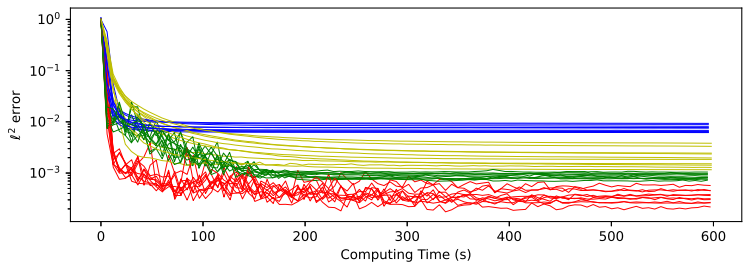}
\caption{\em $\ell^2$ errors v.s. computing time in the comparative example (Red: SelectNet model; Blue: the basic model; green: DRM; yellow: WAN. 10 selected curves for each method).}
\label{Fig_case1_comparison_errors}
\end{figure}

\begin{figure}
\centering
\includegraphics[scale=0.35]{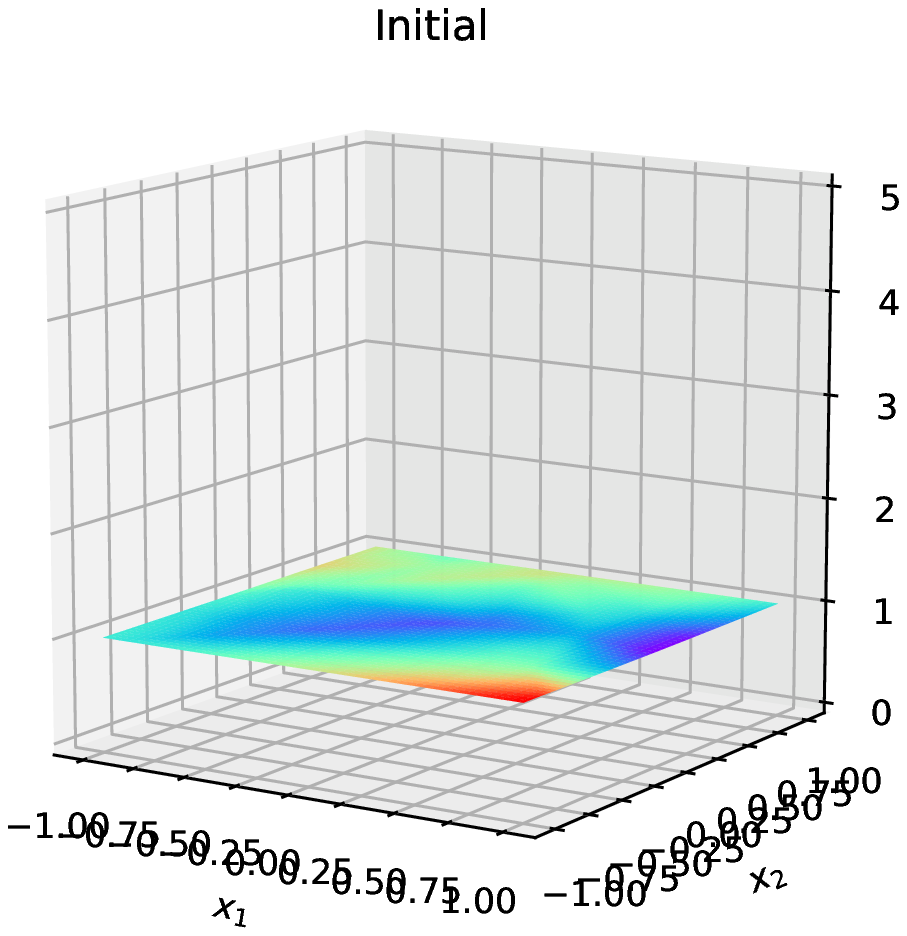}
\includegraphics[scale=0.35]{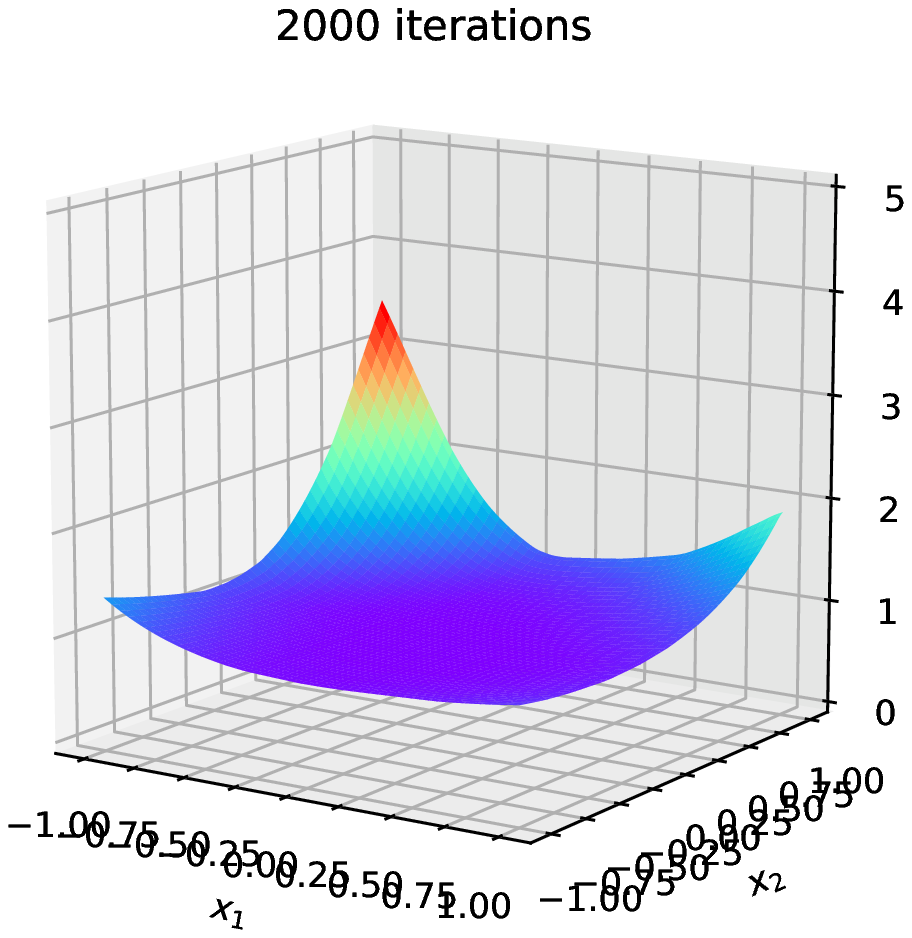}
\includegraphics[scale=0.35]{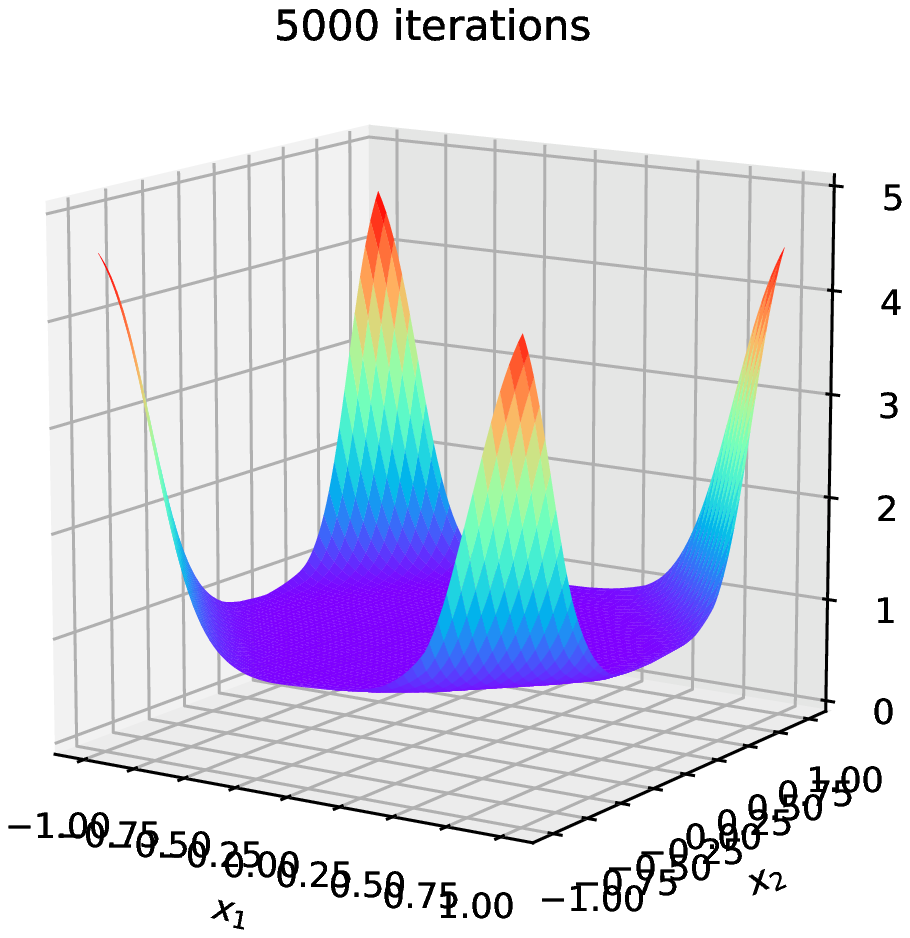}
\includegraphics[scale=0.35]{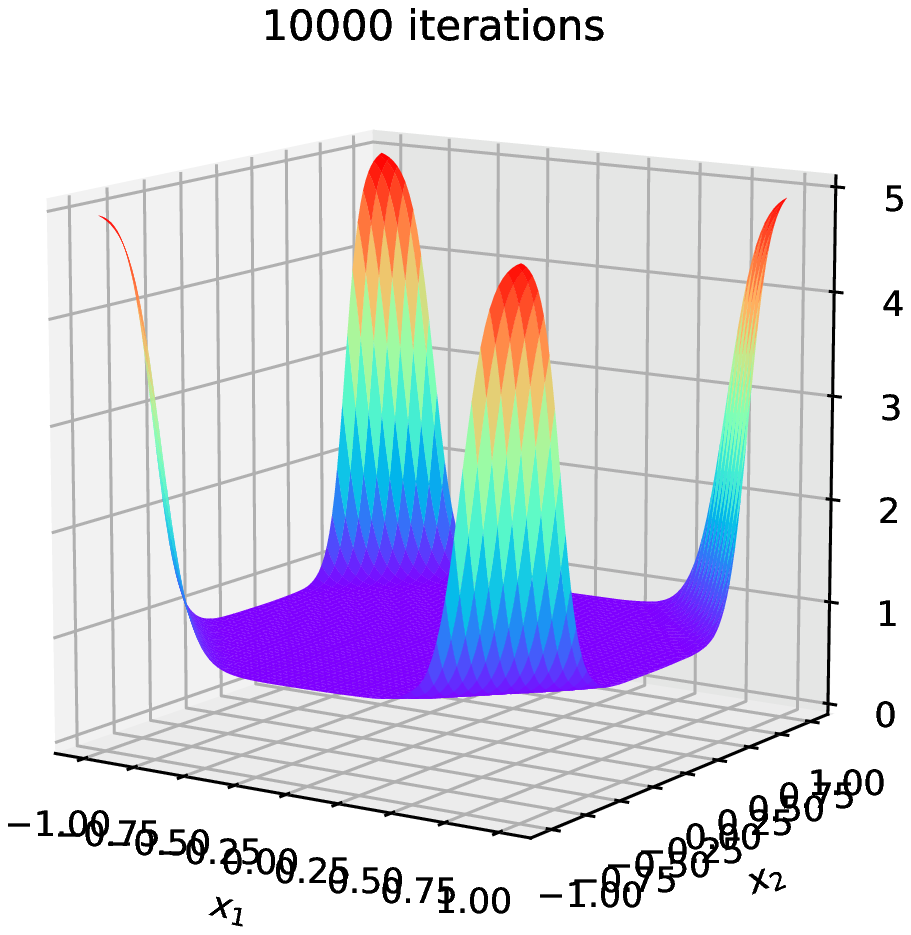}
\caption{\em The evolution of the selection network over time in the comparative example.}
\label{Fig_case1_mesh}
\end{figure}

\subsubsection{Comparison with Binary Weighting}\label{Sec_compare_binary}
To verify that SelectNet is advantageous over other weighting strategies, we also implement the binary weighting method. Namely, in the basic least squares method, we select $p$ ($0<p<1$) training points having larger residuals to be weighted with $w_\tL>1$, and let the other points be weighted with $w_\tS<1$. Specifically, we solve the problem \eqref{34} by
\begin{multline}\label{36}
\underset{\theta}{\min}~ \frac{1}{N_1}\left(\sum_{\Bx\in\mathcal{L}_1}w_\tL|-\Delta u(\Bx;\theta)-1|^2+\sum_{\Bx\in\mathcal{S}_1}w_\tS|-\Delta u(\Bx;\theta)-1|^2\right)\\
+\frac{\lambda}{N_2}\left(\sum_{\Bx\in\mathcal{L}_2}w_\tL|u(\Bx;\theta)|^2+\sum_{\Bx\in\mathcal{S}_2}w_\tS|u(\Bx;\theta)|^2\right),
\end{multline}
where $\{\mathcal{L}_1,\mathcal{S}_1\}$ is a partition of $\{\Bx^1_i\}_{i=1}^{N_1}$ satisfying $|\mathcal{L}_1|=pN_1$, $|\mathcal{S}_1|=(1-p)N_1$, $|-\Delta u(\Bx';\theta)-1|\geq|-\Delta u(\Bx'';\theta)-1|$ for any $\Bx'\in\mathcal{L}_1$ and $\Bx''\in\mathcal{S}_1$; $\{\mathcal{L}_2,\mathcal{S}_2\}$ is a partition of $\{\Bx^2_i\}_{i=1}^{N_2}$ satisfying $|\mathcal{L}_2|=pN_2$, $|\mathcal{S}_2|=(1-p)N_2$, $|u(\Bx';\theta)|\geq| u(\Bx'';\theta)|$ for any $\Bx'\in\mathcal{L}_2$ and $\Bx''\in\mathcal{S}_2$. The binary weights are chosen subject to the following normalization condition
\begin{equation}
w_\tL p+w_\tS(1-p)=1,\quad w_\tL p+w_\tS(1-p)=1.
\end{equation}

As with the preceding tests, we implement the weighting model \eqref{36} with various combinations of parameters for 600 seconds. The means and deviations of the final $\ell^2$ errors are listed in Table \ref{Tab_case1_weighted_errors}. It shows the best combination obtains the mean error $7.375\times10^{-3}$, which is slightly better than the original basic model and much worse than the SelectNet model.

\begin{table}
\centering
\begin{tabular}{|c|c|c|c|}
  \hline
   \multicolumn{4}{|c|}{$p=20\%$}    \\\hline
   $w_\tL/w_\tS$ & 2 & 4 & 8 \\\hline
    Mean of Errors $\mu$ & $8.104\times10^{-3}$ & $8.649\times10^{-3}$ & $9.260\times10^{-3}$ \\\hline
    Standard Deviation $\sigma$ & $8.384\times10^{-4}$ & $1.131\times10^{-3}$ & $1.205\times10^{-3}$ \\\hline
    Coefficient of Variation $\sigma/\mu$ & 10.3\% & 13.1\% & 13.0\% \\\hline\hline
   \multicolumn{4}{|c|}{$p=50\%$}    \\\hline
   $w_\tL/w_\tS$ & 2 & 4 & 8 \\\hline
    Mean of Errors $\mu$ & $7.395\times10^{-3}$ & $7.612\times10^{-3}$ & $7.506\times10^{-3}$ \\\hline
    Standard Deviation $\sigma$ & $1.080\times10^{-3}$ & $1.168\times10^{-3}$ & $1.113\times10^{-3}$ \\\hline
    Coefficient of Variation $\sigma/\mu$ & 14.6\% & 15.3\% & 14.8\% \\\hline\hline
   \multicolumn{4}{|c|}{$p=80\%$}    \\\hline
   $w_\tL/w_\tS$ & 2 & 4 & 8 \\\hline
    Mean of Errors $\mu$ & $7.426\times10^{-3}$ & $\bm{7.375\times10^{-3}}$ & $7.512\times10^{-3}$ \\\hline
    Standard Deviation $\sigma$ & $9.502\times10^{-4}$ & $1.023\times10^{-3}$ & $1.077\times10^{-3}$ \\\hline
    Coefficient of Variation $\sigma/\mu$ & 12.8\% & 13.9\% & 14.3\% \\\hline
\end{tabular}
\caption{\em Means and standard deviations of the $\ell^2$ errors obtained within 600 seconds by binary weighting in the comparative example. (totally 50 trials for each combination)}
\label{Tab_case1_weighted_errors}
\end{table}

\subsection{High-dimensional Examples}
In the second experiment, we will implement the basic and SelectNet models in a series of high-dimensional examples ($d\geq10$) to reflect the advantage of using SelectNet. Note from the preceding comparative experiment that SelectNet can obtain much smaller error means than the basic model, which overwhelms the error deviations. Therefore, considering the long time spent in high-dimensional problems, we only implement both models for once in each case to present the results in the paper.

Since in high-dimensional cases, most of the random points following a uniform distribution are near the boundary, we take an annularly uniform strategy instead of uniform sampling. Specifically, for a high-dimensional unit circle, we divide it into $N_a$ annuli $\{k/N_a<|\Bx|<(k+1)/N_a\}_{k=0}^{N_a-1}$ and generates $N_1/N_a$ samples uniformly in each annulus. In the following experiments, we choose $N_a=10$. This sampling strategy is applied in generating interior training points and testing points. For generating boundary training points, we still use uniform sampling.

\subsubsection{Elliptic Equations with Low-Regularity Solutions}\label{Sec_01}
First, let us consider the nonlinear elliptic equation inside a bounded domain
\begin{equation}
\begin{split}
-\nabla\cdot (a(x)\nabla u) + |\nabla u|^2&=f(x),\quad\text{in~}\Omega:=\{x:|x|<1\},\\
u&=g(x),\quad\text{on}~\partial\Omega,
\end{split}
\end{equation}
with $a(x)=1+\frac{1}{2}|x|^2$. In this case, we specify the exact solution by
\begin{equation}\label{14}
u(x)=\sin(\frac{\pi}{2}(1-|x|)^{2.5}),
\end{equation}
whose first derivative is singular at the origin and the third derivative is singular on the boundary. Note the problem is nonlinear if $\mu\neq0$. We solve the high-dimensional nonlinear problem for $d=10$, $20$ and $100$. The errors obtained by the basic and SelectNet models with 20000 iterations are listed in Table \ref{Tab_case5_errors}. Since the basic model costs less time for one iteration, we also list the errors obtained by SelectNet with the same computing time as the basic model for comparison. The curves of error decay versus iterations are shown in Fig. \ref{Fig_case5_errors}. From these results, it is observed both models are effective on the nonlinear elliptic problem of all dimensions, but SelectNet has a clearly better performance than the basic model: its accuracy is one-digit better than the basic model. Besides, we present in Fig. \ref{Fig_case5_mesh} the following surfaces at $(x_1,x_2)$-slice
\begin{itemize}
  \item the numerical solution: $u(x_1,x_2,0,\cdots,0;\theta)$
  \item the modulus of the numerical residual error: $|\mathcal{D}u(x_1,x_2,0,\cdots,0;\theta)-f(x_1,x_2,0,\cdots,0)|$
  \item the selection network: $\phi_\ts'(x_1,x_2,0,\cdots,0;\theta_\ts')$
\end{itemize}
for the $20$-dimensional case. It shows that the residual error accumulates near the origin due to its low regularity. On the other hand, the selection network attains its peak at the origin, implying that training points are highly weighted near the origin where the residual error is mainly distributed. Note that the selection network is not supported locally near the low-regularity point, which means that the selection network will not make the training of the solution network focus on the low-regularity point only.

\begin{table}
\centering
\begin{tabular}{|c|c|c|c|}
  \hline
  Dimension & SelectNet & SelectNet$^*$ & Basic \\\hline
    $d=10$& $7.944\times10^{-4}$ & $8.089\times10^{-4}$ & $3.193\times10^{-3}$ \\\hline
    $d=20$& $9.584\times10^{-4}$ & $1.241\times10^{-3}$ & $1.707\times10^{-2}$ \\\hline
    $d=100$ & $9.257\times10^{-3}$ & $1.004\times10^{-2}$ & $1.862\times10^{-1}$ \\\hline
\end{tabular}
\caption{\em $\ell^2$ errors obtained by various models in the nonlinear elliptic example. (``SelectNet" and ``Basic" denote the final errors obtained by SelectNet and basic models with $20000$ iterations; ``SelectNet$^*$" denotes the error obtained by SelectNet model with the same computing time as 20000 iterations of basic model, the same below)}
\label{Tab_case5_errors}
\end{table}

\begin{figure}
\centering
\subfloat[$d=10$]{
\includegraphics[scale=0.5]{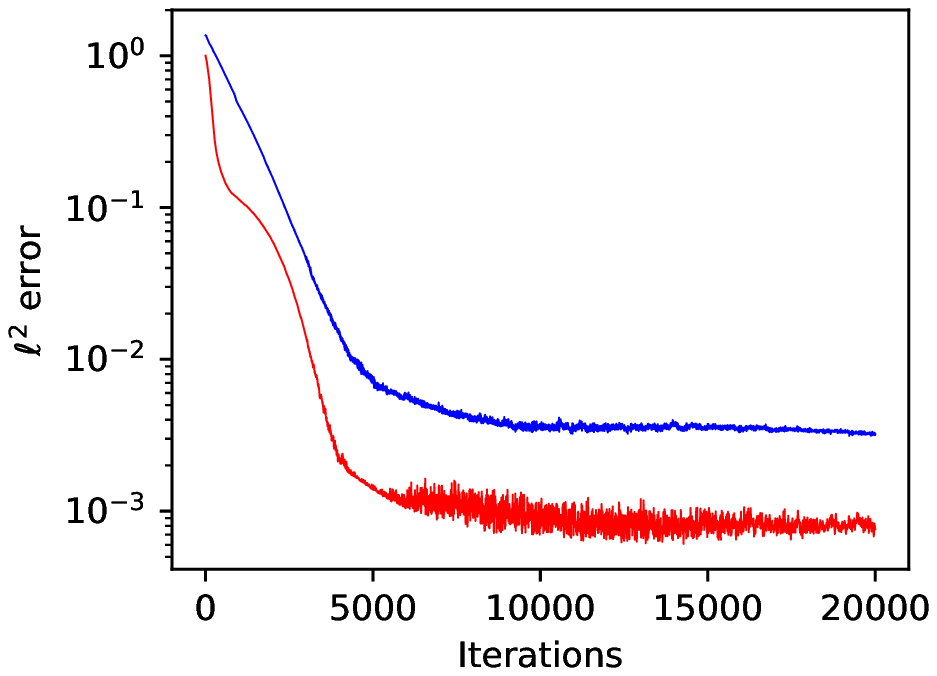}}
\subfloat[$d=20$]{
\includegraphics[scale=0.5]{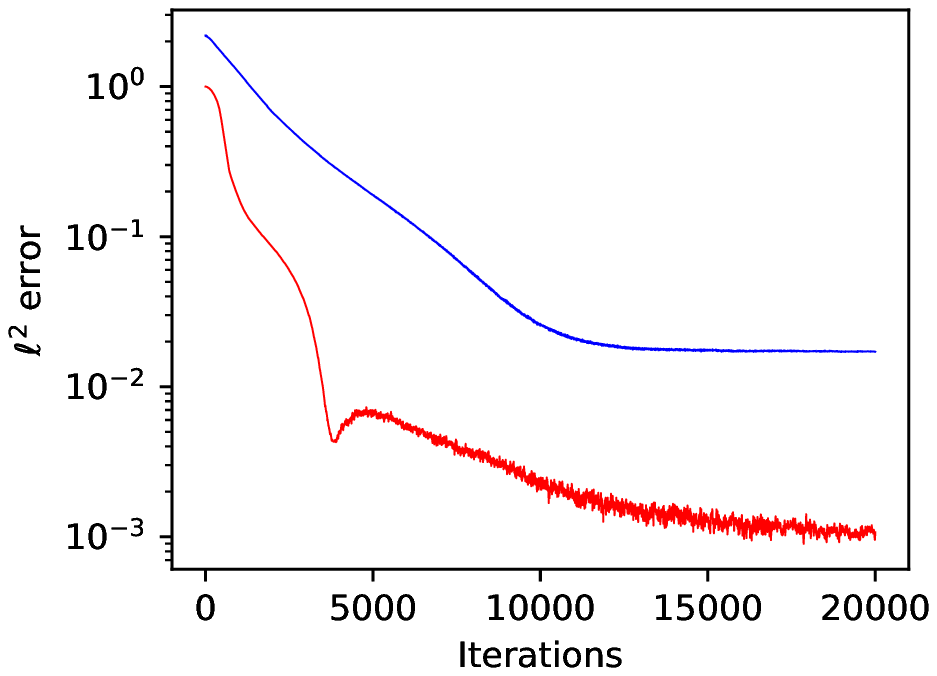}}
\subfloat[$d=100$]{
\includegraphics[scale=0.5]{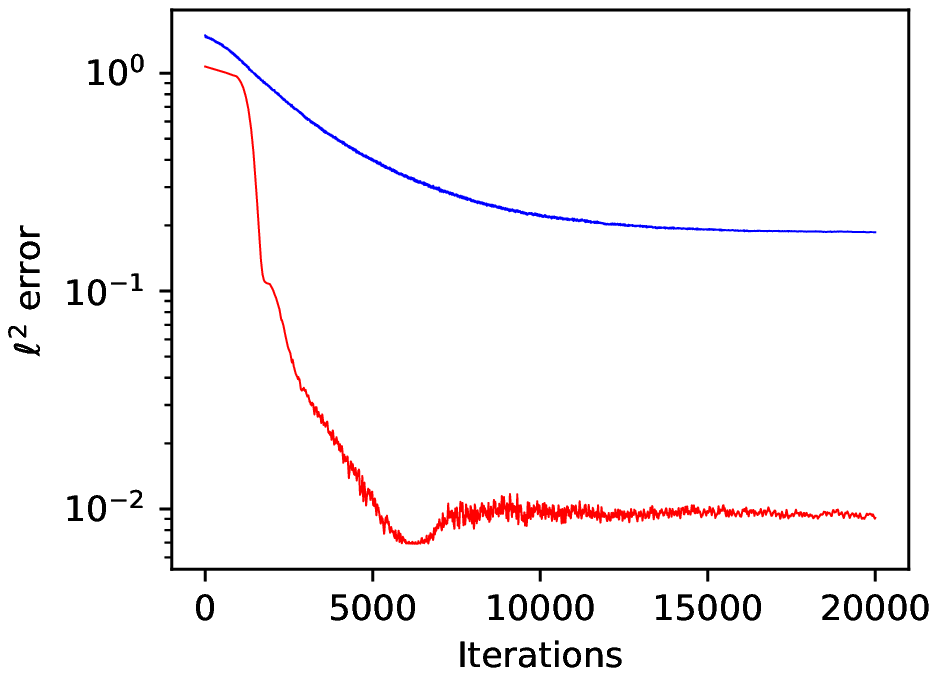}}
\caption{\em $\ell^2$ errors v.s. iterations in the nonlinear elliptic example (Red: SelectNet model; Blue: the basic model).}
\label{Fig_case5_errors}
\end{figure}

\begin{figure}
\centering
\includegraphics[scale=0.4]{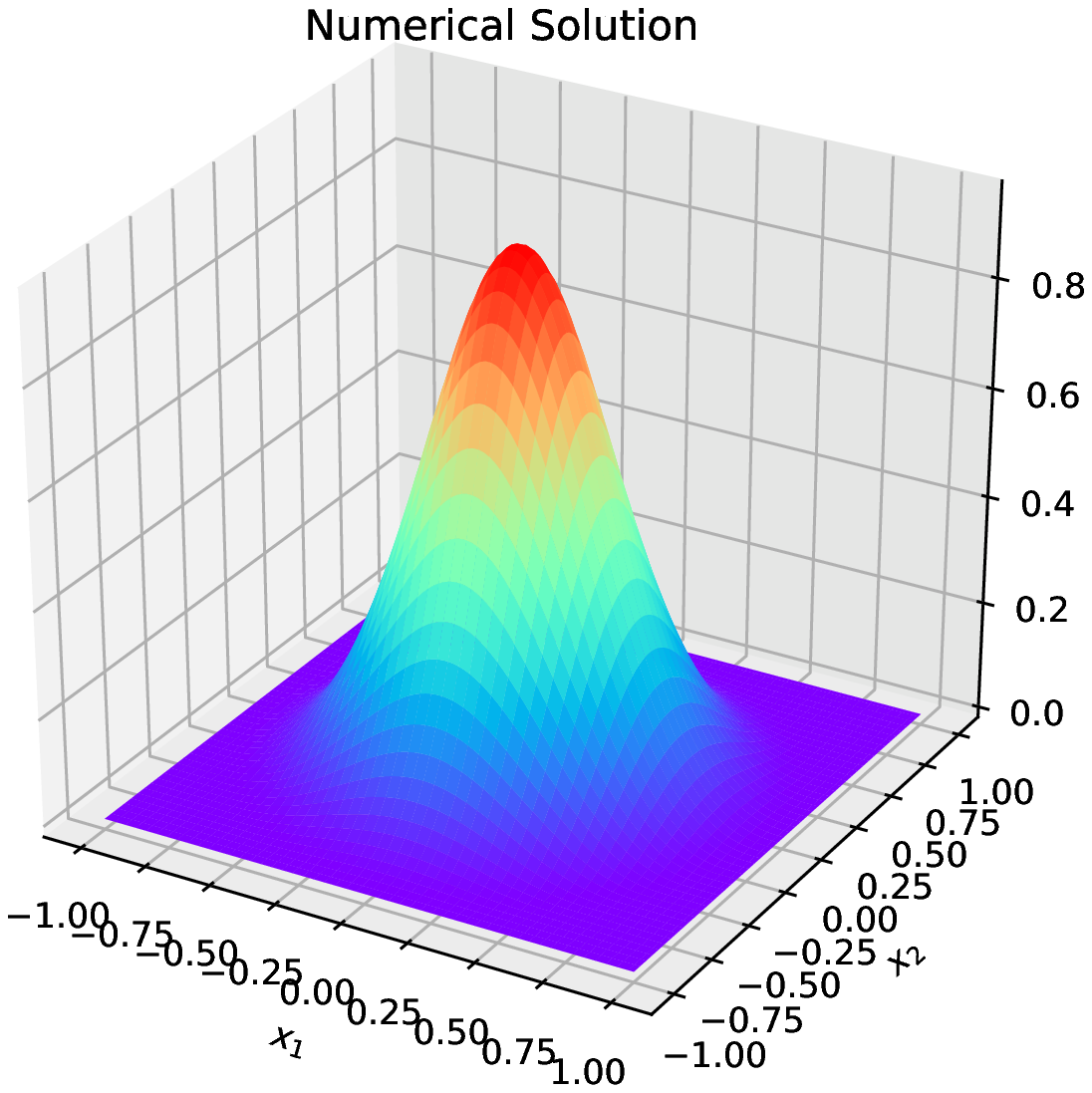}
\includegraphics[scale=0.4]{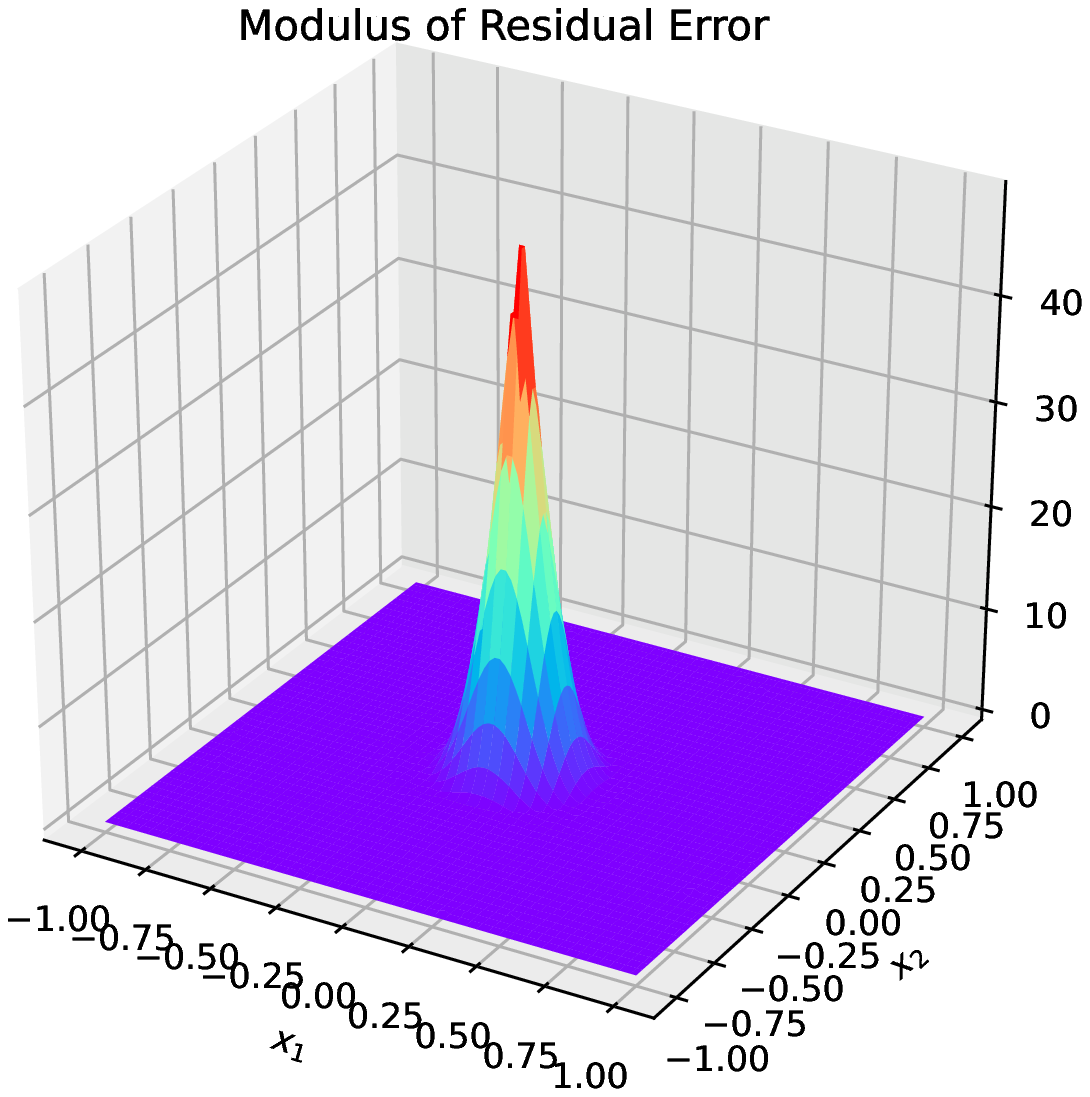}
\includegraphics[scale=0.4]{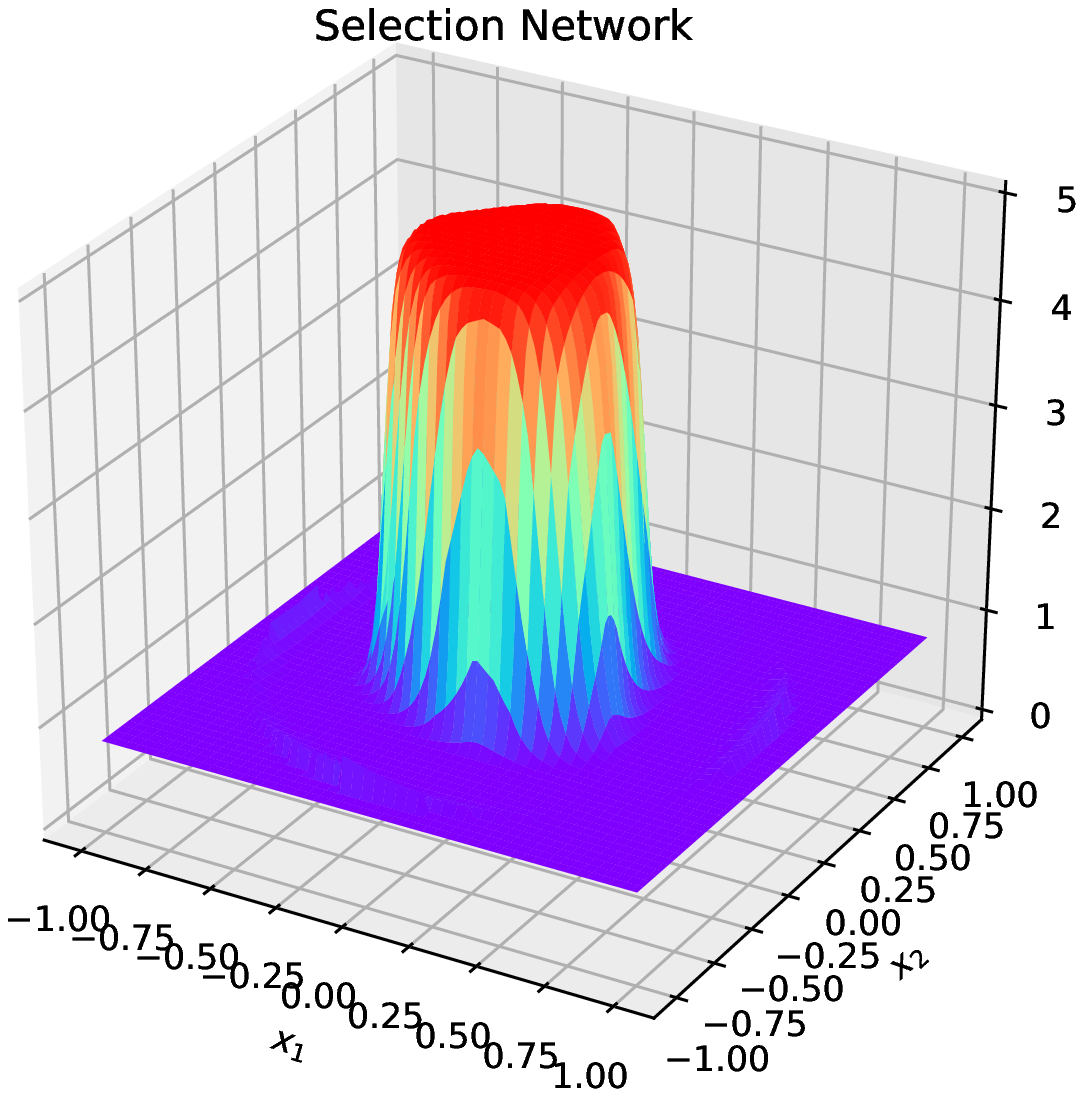}
\caption{\em The $(x_1,x_2)$-surfaces of the numerical solution, the modulus of residual errors and selection network by SelectNet (d=20) in the nonlinear elliptic example.}
\label{Fig_case5_mesh}
\end{figure}

\subsubsection{Linear Parabolic Equations}
In this example, SelectNet is tested on the following initial boundary value problem of the linear parabolic equation
\begin{equation}
\begin{split}
\partial_t u(x,t)-\nabla_x\cdot(a(x)\nabla_x u(x,t))&=f(x,t),\quad\text{in~}Q:=\Omega\times(0,1),\\
u(x,t)&=g(x),\quad\text{on}~\partial\Omega\times(0,1),\\
u(x,0)&=h(x),\quad\text{in}~\Omega,\\
\end{split}
\end{equation}
where $a(x)=1+\frac{1}{2}|x|$ and $\Omega:=\{x:|x|<1\}$. The exact solution is set by
\begin{equation}
u(x,t)=\exp(|x|\sqrt{1-t}).
\end{equation}
Note $u$ is at most $C^0$ smooth at $t=1$ and $|x|=0$. In the SelectNet model, time-discretization schemes are not utilized. Instead, we regard $t$ as an extra spatial variable of the problem. Hence the problem domain $\Omega\times(0,1)$ is an analog of a hypercylinder, and the ``boundary" value is specified in the bottom $\Omega\times\{t=0\}$ and the side $\partial\Omega\times(0,1)$. This example is tested for $d=10$, $20$ and $100$, by evaluating the relative $\ell^2$ error in $\Omega\times(0,1)$. The errors of the basic and SelectNet models are listed in Table \ref{Tab_case7_errors}. It is clearly shown SelectNet still obtains smaller errors than the basic model with the same number of iterations or computing time. In Fig. \ref{Fig_case7_errors} the curves of error decay are presented, and in Fig. \ref{Fig_case7_mesh} the $(t,x_1)$-surfaces of the numerical solution, the modulus of the residual errors and selection network for $d=20$ are displayed, from that we can observe the residual error is mainly distributed near the singular point $\Bx=0$ and the terminal slice $t=1$. Accordingly, the selection network takes its maximum in this region.

\begin{table}
\centering
\begin{tabular}{|c|c|c|c|}
  \hline
  Dimension & SelectNet & SelectNet$^*$ & Basic \\\hline
  $d=10$ & $1.490\times10^{-2}$ & $1.502\times10^{-2}$ & $3.531\times10^{-2}$ \\\hline
  $d=20$ & $2.990\times10^{-2}$ & $3.000\times10^{-2}$ & $8.748\times10^{-2}$ \\\hline
  $d=100$ & $6.302\times10^{-2}$ & $6.268\times10^{-2}$ & $1.357\times10^{-1}$ \\\hline
\end{tabular}
\caption{\em $\ell^2$ errors obtained by various models in the linear parabolic example.}
\label{Tab_case7_errors}
\end{table}

\begin{figure}
\centering
\subfloat[$d=10$]{
\includegraphics[scale=0.5]{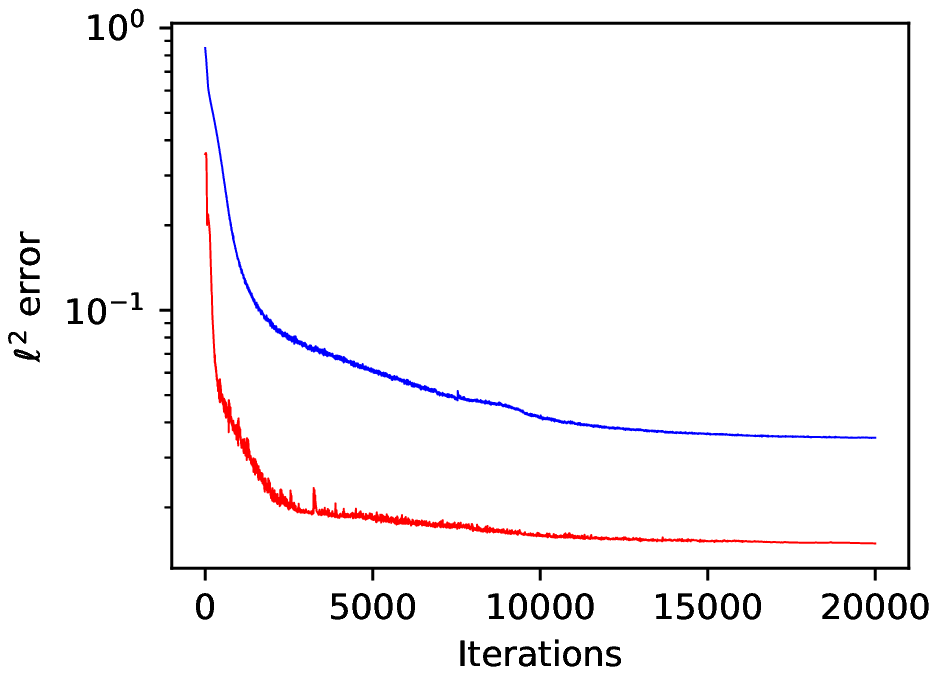}}
\subfloat[$d=20$]{
\includegraphics[scale=0.5]{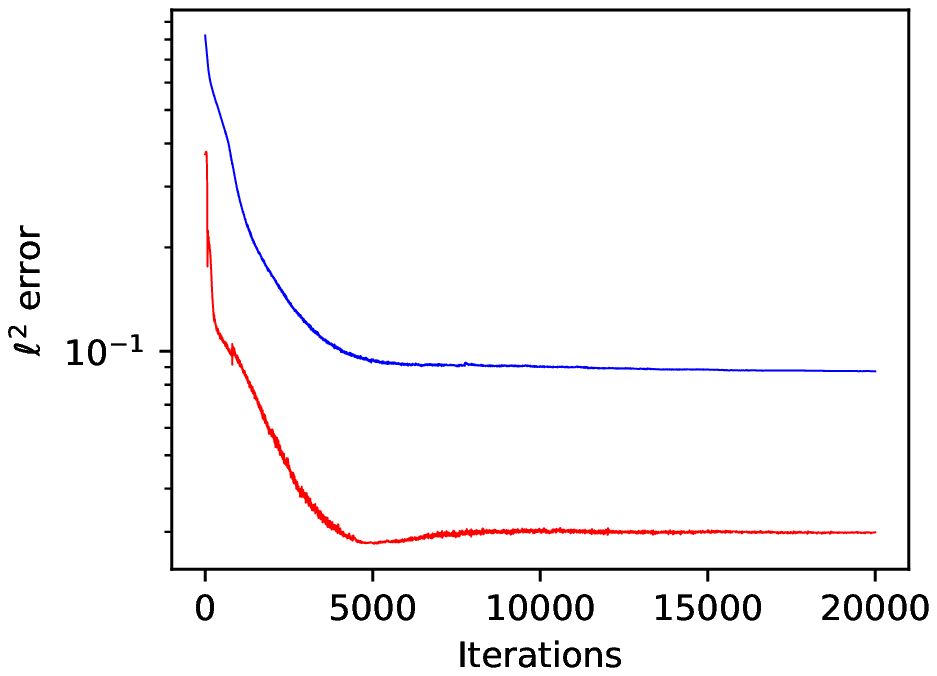}}
\subfloat[$d=100$]{
\includegraphics[scale=0.5]{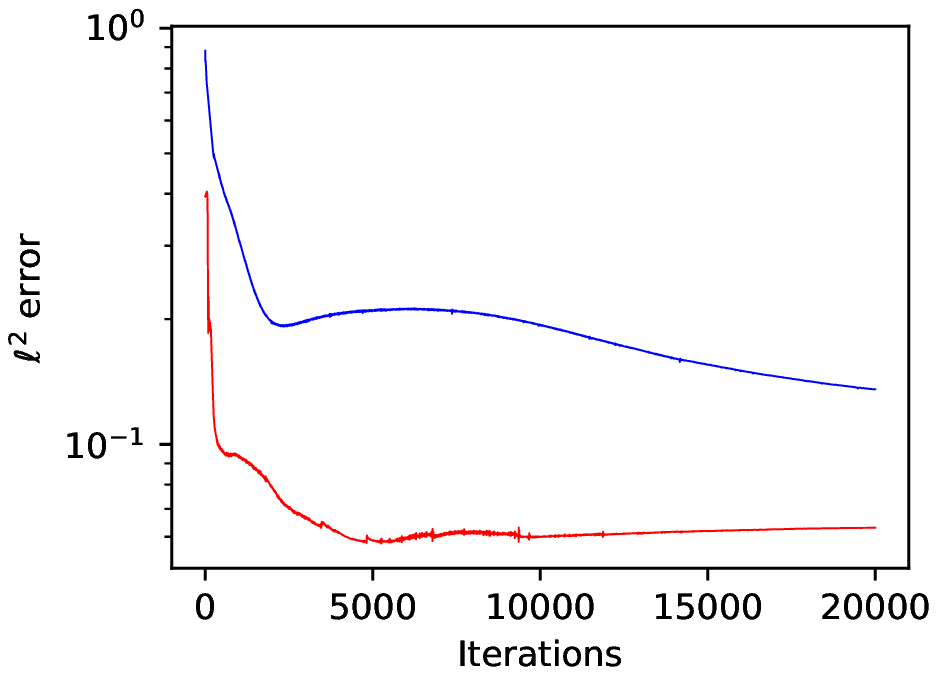}}
\caption{\em $\ell^2$ errors v.s. iterations in the linear parabolic example (Red: SelectNet model; Blue: the basic model).}
\label{Fig_case7_errors}
\end{figure}

\begin{figure}
\centering
\includegraphics[scale=0.4]{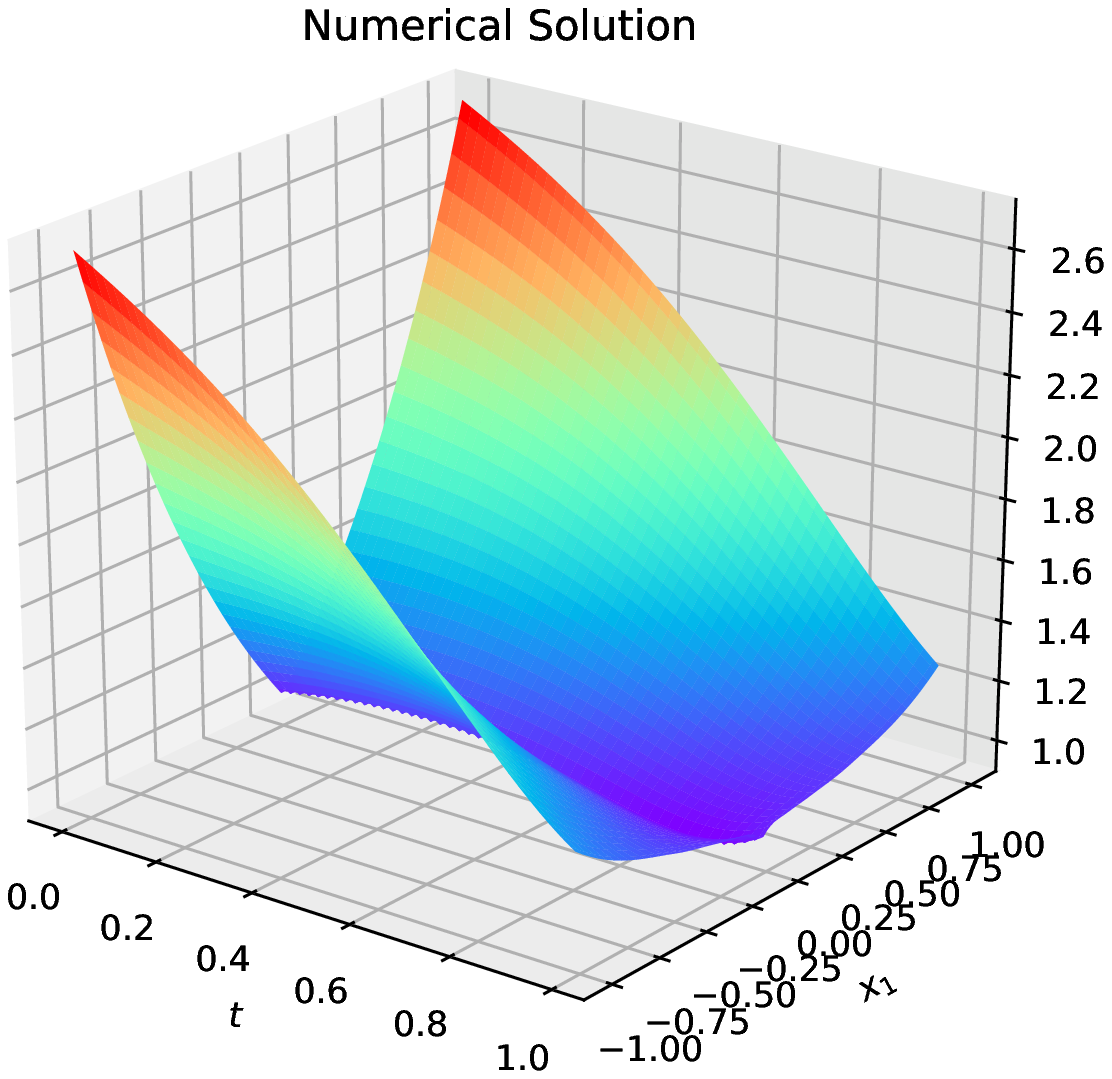}
\includegraphics[scale=0.4]{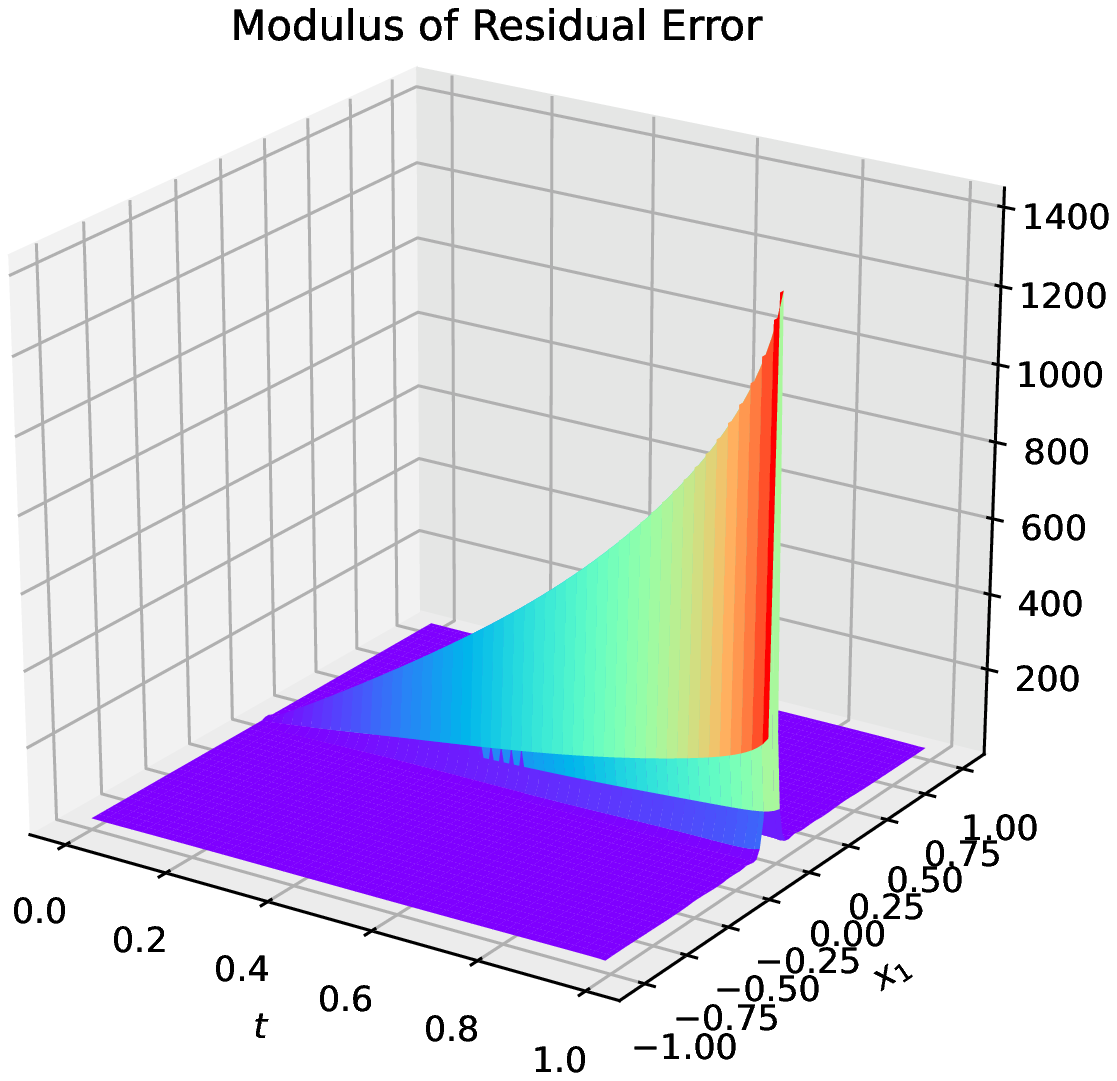}
\includegraphics[scale=0.4]{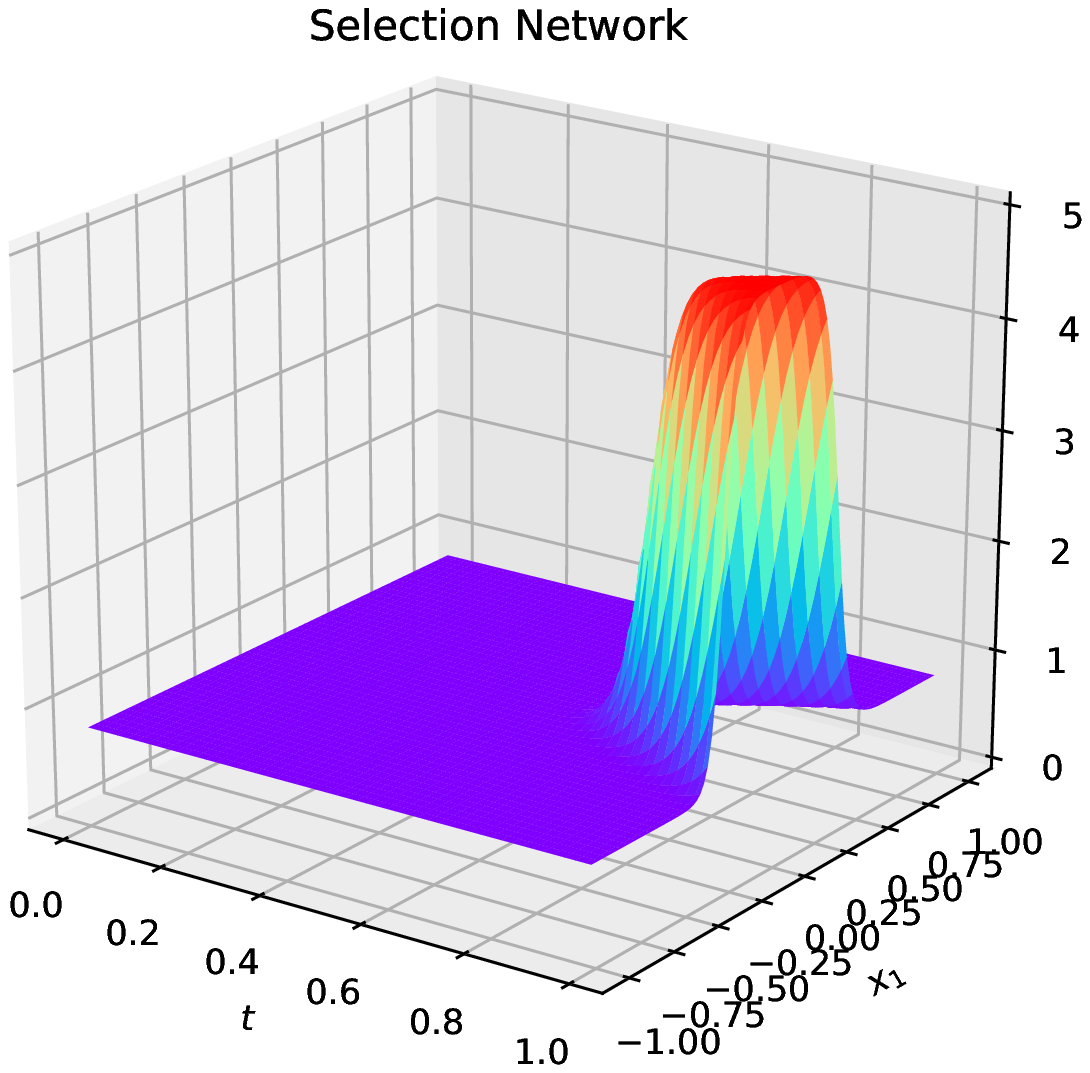}
\caption{\em The $(t,x_1)$-surfaces of the numerical solution, the modulus of residual errors and selection network by SelectNet (d=20) in the linear parabolic example.}
\label{Fig_case7_mesh}
\end{figure}

\subsubsection{Allen-Cahn Equation}
In this example, we test SelectNet model on the following 100-dimensional Allen-Cahn equation
\begin{equation}
\begin{split}
\partial_t u(x,t)-\Delta_x u(x,t)-u(x,t)+u^3(x,t)&=f(x,t),\quad\text{in~}Q:=\Omega\times(0,1),\\
u(x,t)&=g(x),\quad\text{on}~\partial\Omega\times(0,1),\\
u(x,0)&=h(x),\quad\text{in}~\Omega,\\
\end{split}
\end{equation}
where $a(x)=1+\frac{1}{2}|x|$ and $\Omega:=\{x:|x|<1\}$. Note the Allen-Cahn equation is a nonlinear parabolic equation. The exact solution is set as
\begin{equation}\label{15}
u(x,t)=e^{-t}\sin(\frac{\pi}{2}(1-|x|)^{2.5}).
\end{equation}
The errors obtained by SelectNet model and the basic model with 20000 iterations are $6.358\times10^{-3}$ and $3.347\times10^{-2}$, respectively. And the SelectNet error obtained with the same computing time as the basic model is $6.218\times10^{-3}$. The error curves versus iterations are shown in Fig. \ref{Fig_case9_errors}. It can be seen in the figure the error curve of the SelectNet decays faster to lower levels than the basic model. Moreover, the $(t,x_1)$-surface of the numerical solution, the modulus of residual errors and selection network are shown in Fig. \ref{Fig_case9_mesh}, from that we can observe the selection network takes its maximum near the singular point $\Bx=0$ and the initial slice $t=0$, where the highest residual error is located.

\begin{figure}
\centering
\subfloat[$d=100$]{
\includegraphics[scale=0.5]{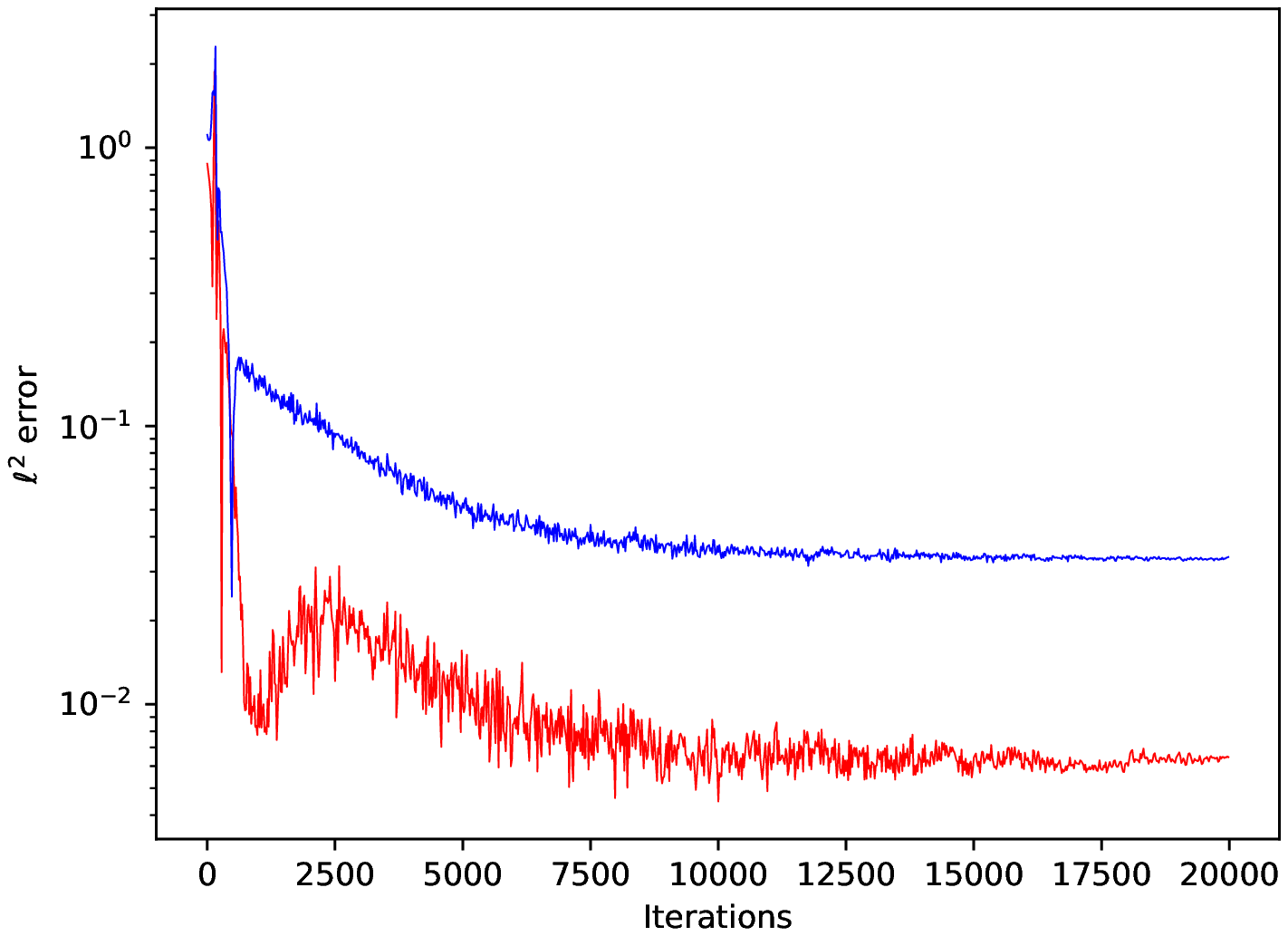}}
\caption{\em $\ell^2$ errors v.s. iterations in the Allen-Cahn example (Red: SelectNet model; Blue: the basic model).}
\label{Fig_case9_errors}
\end{figure}

\begin{figure}
\centering
\includegraphics[scale=0.4]{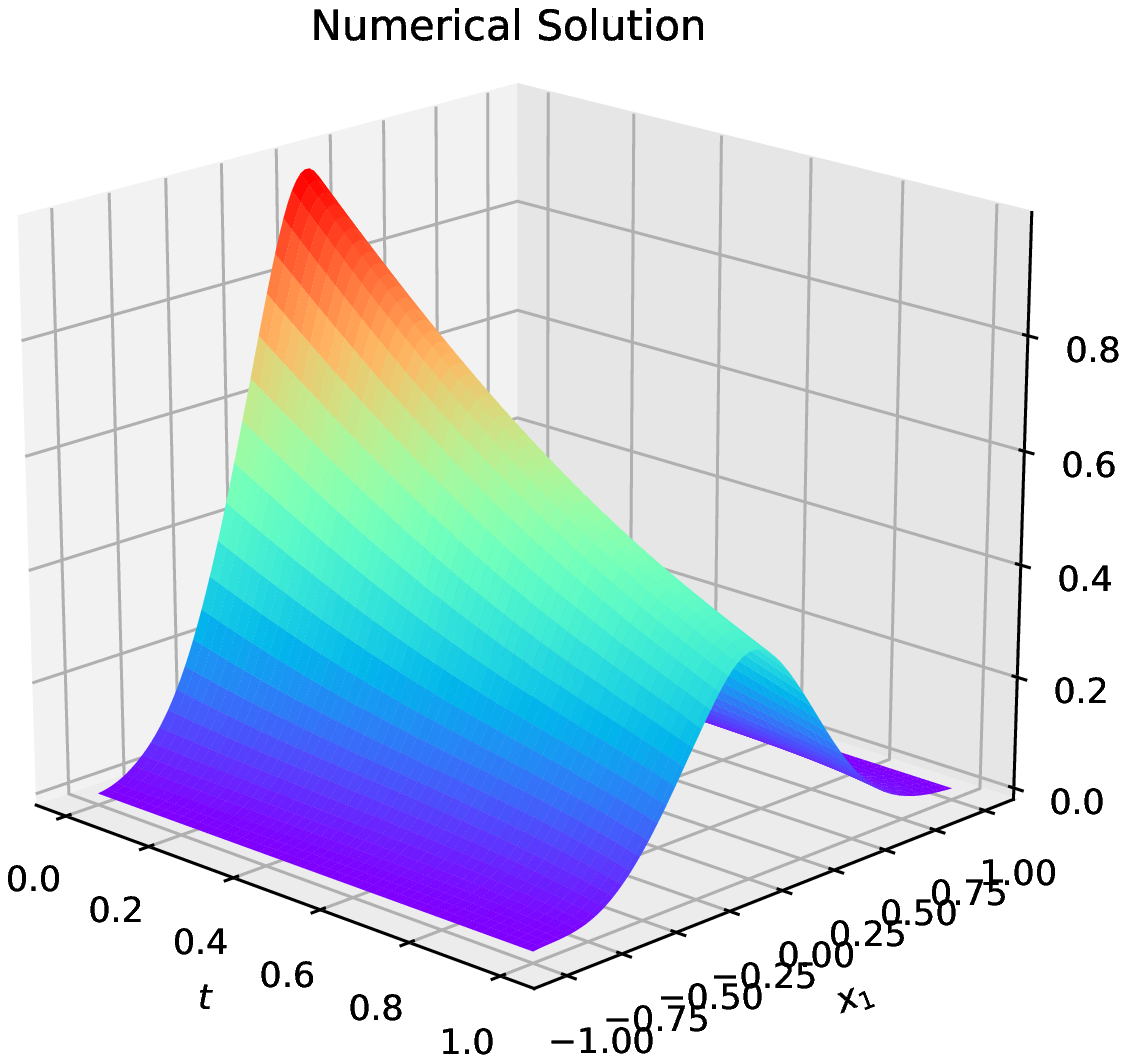}
\includegraphics[scale=0.4]{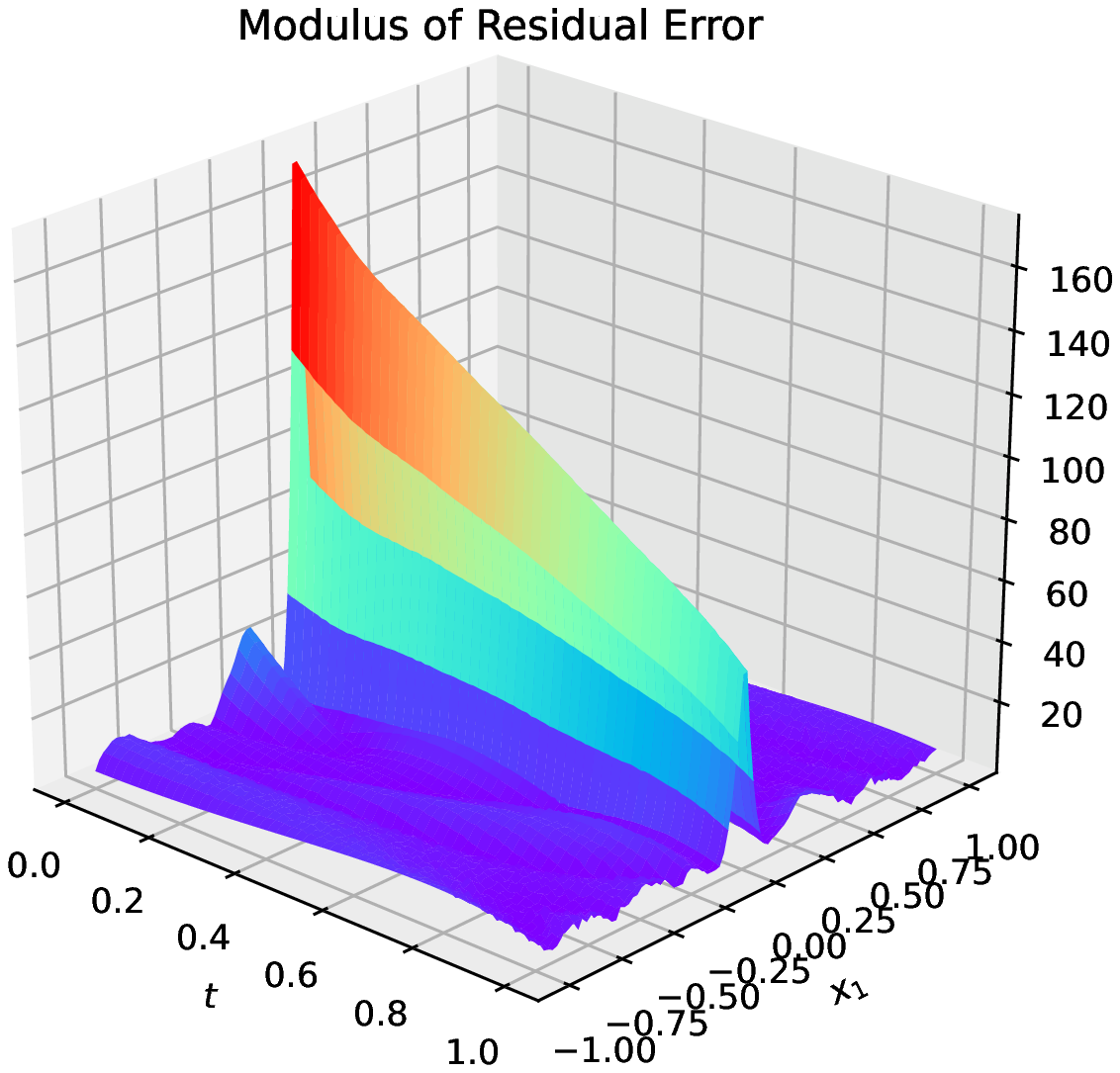}
\includegraphics[scale=0.4]{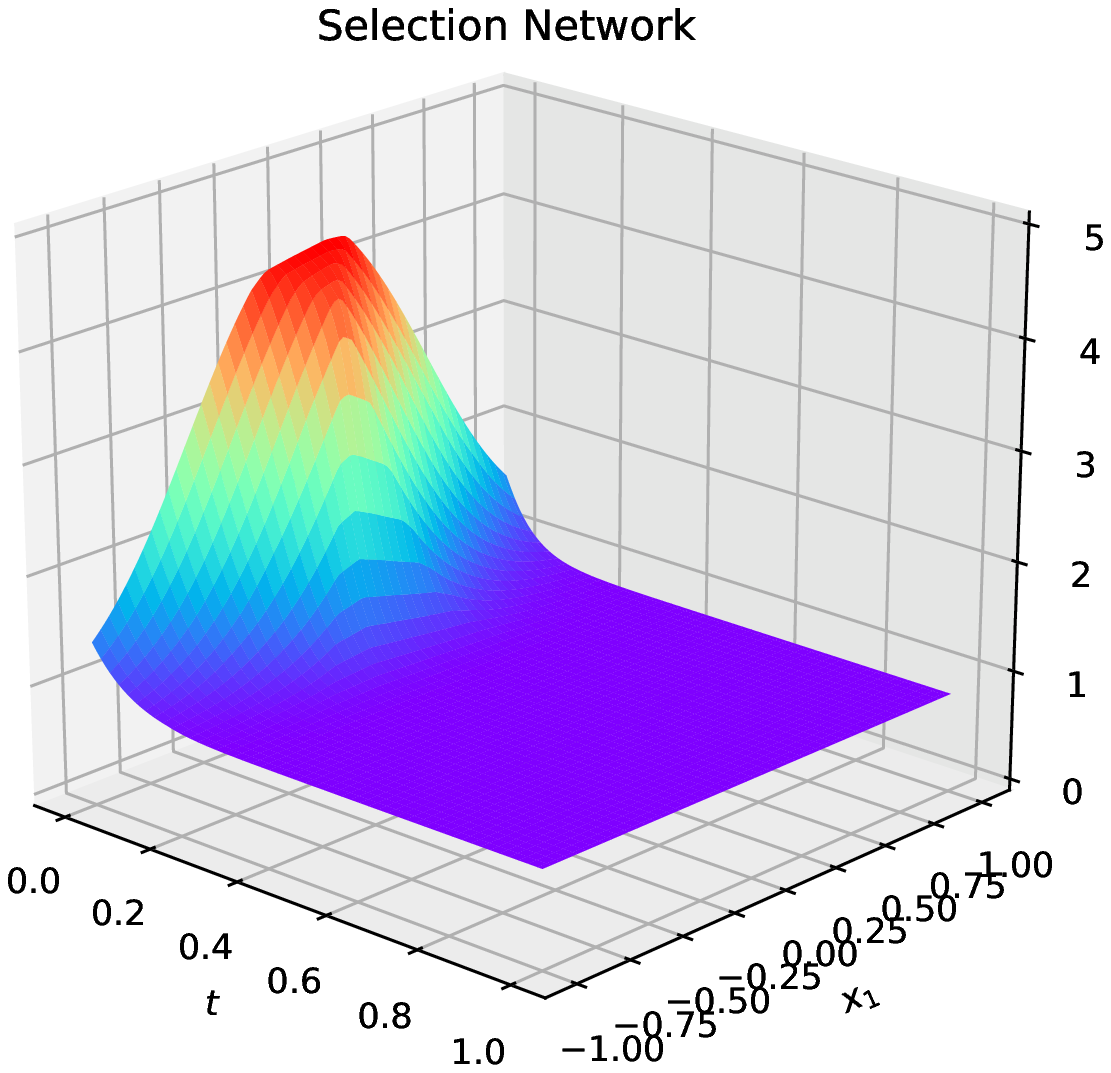}
\caption{\em The $(t,x_1)$-surfaces of the numerical solution, the modulus of residual errors and selection network by SelectNet in the Allen-Cahn example.}
\label{Fig_case9_mesh}
\end{figure}

\subsubsection{Hyperbolic Equations}
In the last example, we test SelectNet by solving the initial boundary value problem of the hyperbolic (wave) equation, which is given by
\begin{equation}
\begin{split}
&\frac{\partial^2 u(x,t)}{\partial t^2}-\Delta_x u(x,t)=f(x,t),\text{~in~}\Omega\times(0,1),\\
&u(x,t)=g_0(x,t),\text{~on~}\partial\Omega\times(0,1),\\
&u(x,0)=h_0(x),\quad\frac{\partial u(x,0)}{\partial t}=h_1(x)\text{~in~}\Omega,
\end{split}
\end{equation}
with $\Omega:=\{x:|x|<1\}$ and exact solution is set by
\begin{equation}
u(x,t)=\left(\exp(t^2)-1\right)\sin(\frac{\pi}{2}(1-|x|)^{2.5}).
\end{equation}
Same as in preceding examples, we solve the problem of $d=10$, $20$ and $100$ and compute the relative $\ell^2$ errors of the basic and SelectNet models. The obtained errors are listed in Table \ref{Tab_case6_errors}, which demonstrates the SelectNet still converges faster than the basic model (especially when $d$ is higher), obtaining smaller errors. Also, we display the curves of error decay in Fig. \ref{Fig_case6_errors}, and the $(t,x_1)$-surfaces of the numerical results when $d=20$ in Fig. \ref{Fig_case6_mesh}. The results in the examples of parabolic and hyperbolic equations imply our proposed model works successfully for time-dependent problems.

\begin{table}
\centering
\begin{tabular}{|c|c|c|c|}
  \hline
  Dimension & SelectNet & SelectNet$^*$ & Basic \\\hline
    $d=10$& $1.671\times10^{-2}$ & $1.701\times10^{-2}$ & $5.200\times10^{-2}$ \\\hline
    $d=20$& $3.281\times10^{-2}$ & $3.292\times10^{-2}$ & $9.665\times10^{-2}$ \\\hline
    $d=100$& $6.319\times10^{-2}$ & $6.351\times10^{-2}$ & $3.089\times10^{-1}$ \\\hline
\end{tabular}
\caption{\em Final $\ell^2$ errors obtained by various models in the hyperbolic example.}
\label{Tab_case6_errors}
\end{table}

\begin{figure}
\centering
\subfloat[$d=10$]{
\includegraphics[scale=0.5]{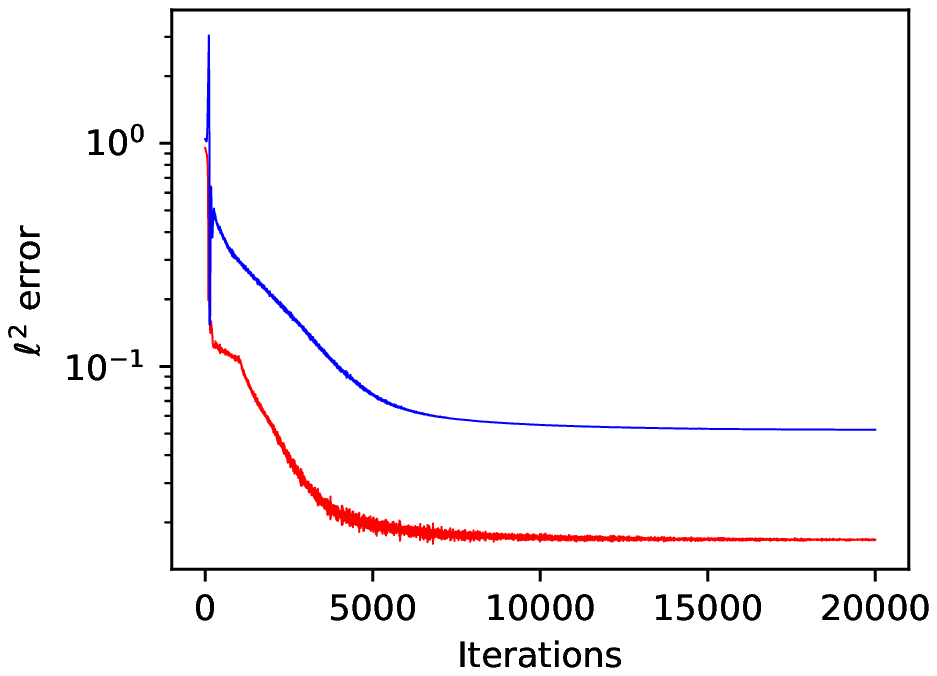}}
\subfloat[$d=20$]{
\includegraphics[scale=0.5]{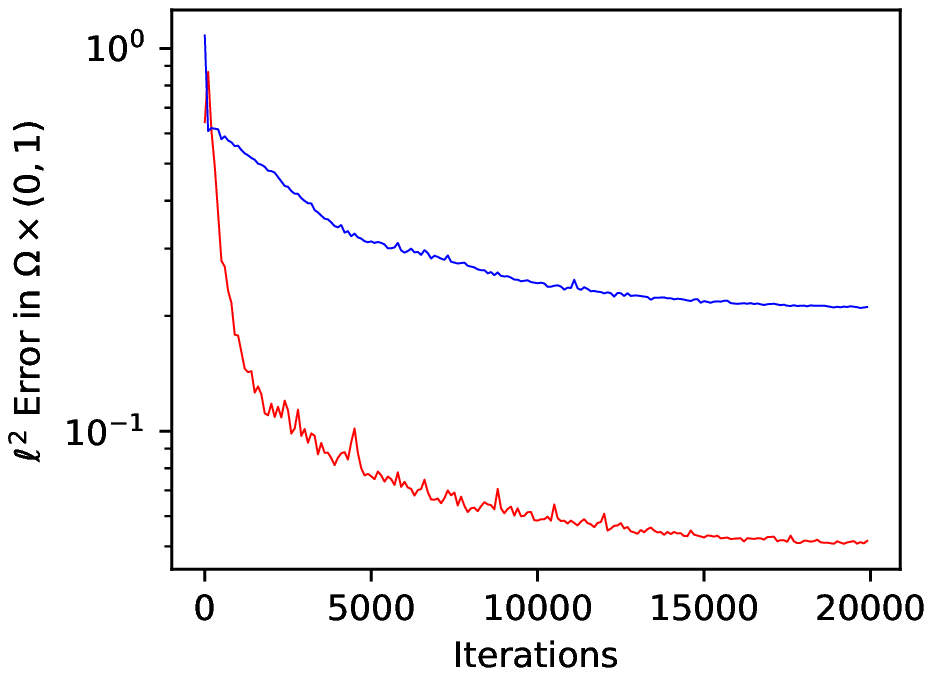}}
\subfloat[$d=100$]{
\includegraphics[scale=0.5]{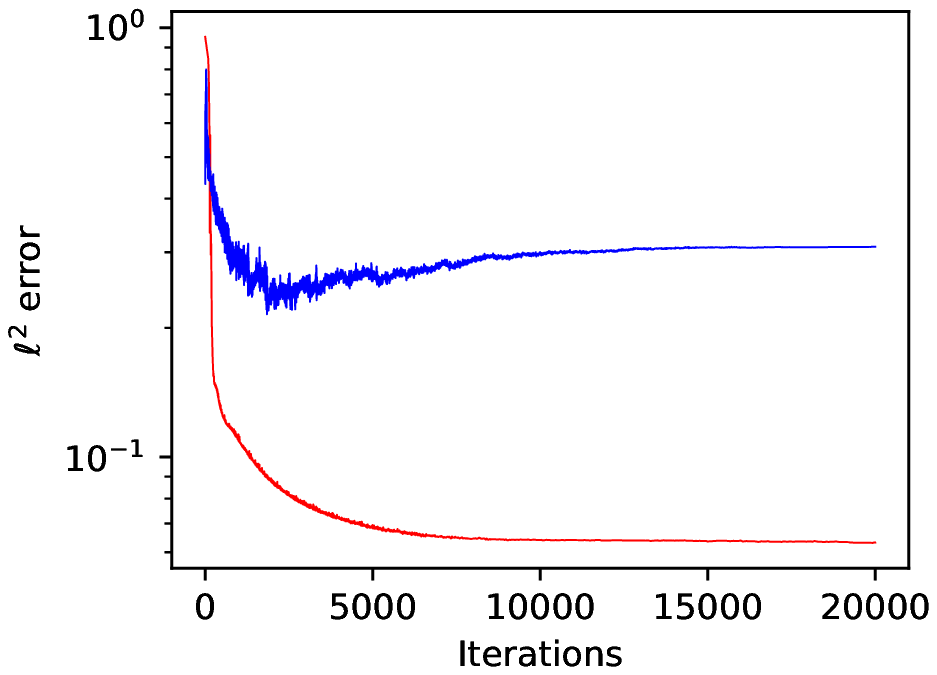}}
\caption{\em $\ell^2$ errors v.s. iterations in the hyperbolic example (Red: SelectNet model; Blue: the basic model).}
\label{Fig_case6_errors}
\end{figure}

\begin{figure}
\centering
\includegraphics[scale=0.4]{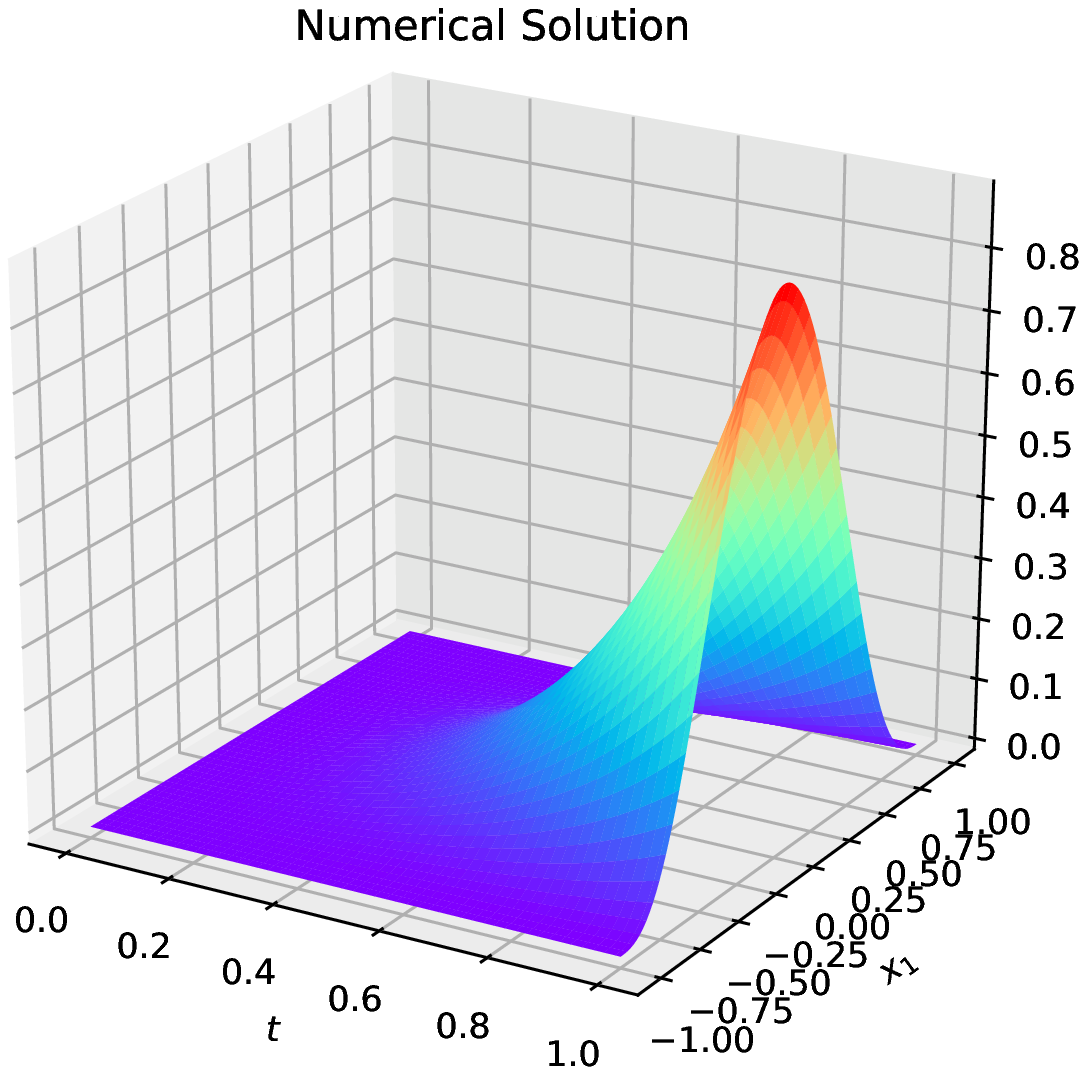}
\includegraphics[scale=0.4]{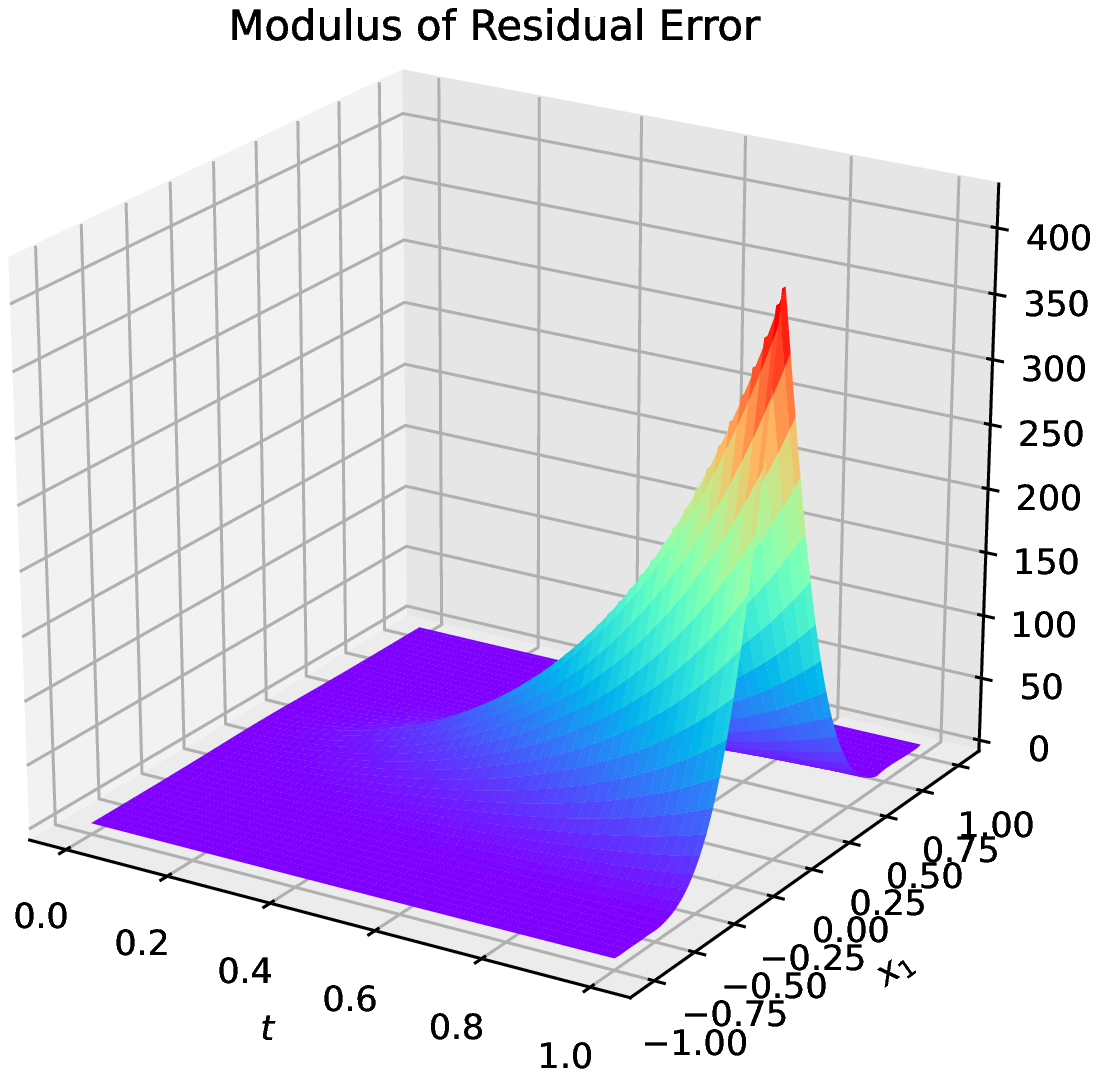}
\includegraphics[scale=0.4]{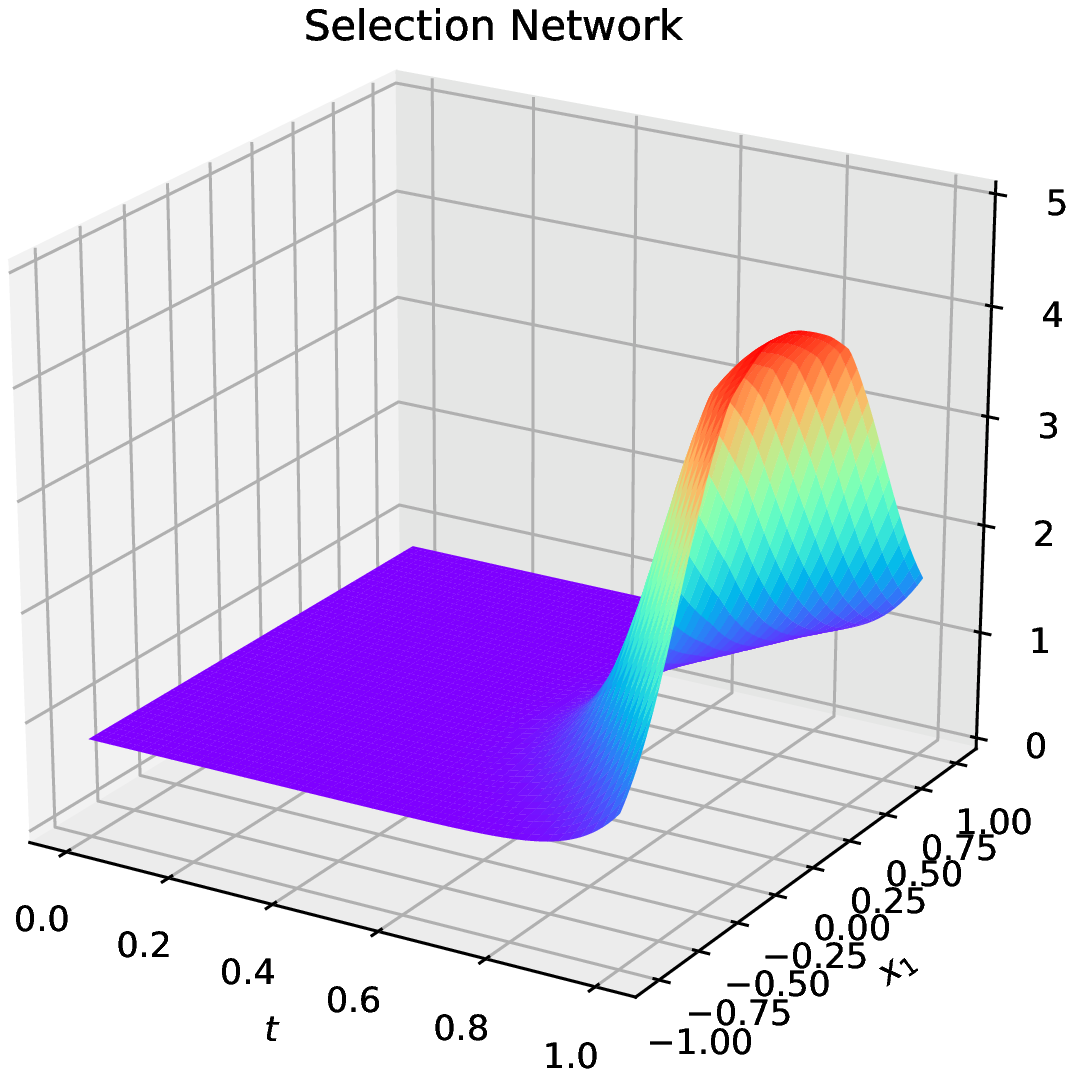}
\caption{\em The $(t,x_1)$-surfaces of the numerical solution, the modulus of residual errors and selection network by SelectNet (d=20) in the hyperbolic example.}
\label{Fig_case6_mesh}
\end{figure}

\section{Conclusion}
In this work, we improve the network-based least squares models on generic PDEs by introducing a selection network for selected sampling in the optimization process. The objective is to place higher weights on the sampling points having larger point-wise residual errors, and correspondingly we propose the SelectNet model that is a min-max optimization. In the implementation, both the solution and selection functions are approximated by feedforward neural networks, which are trained alternatively in the algorithm. The proposed SelectNet framework can solve high-dimensional PDEs that are intractable by traditional PDE solvers.

In the numerical examples, it is demonstrated the proposed SelectNet model works effectively for elliptic, parabolic, and hyperbolic equations, even if in the case of nonlinear equations. Furthermore, numerical results show that the proposed model outperforms the basic least squares model. In the problems with low-regularity solutions, SelectNet will focus on the region with larger errors automatically, finally improving the speed of convergence.

In this paper, we apply neural networks with piecewise polynomial functions as activation functions. If the floor, ReLU, Sign, and exponential functions are used as activation functions, \cite{Shen4,Shen6} showed that deep network approximation has no curse of dimensionality in the approximation error for H{\"o}lder continuous functions. But unfortunately, efficient numerical algorithms for these networks are still not available yet. It is interesting to explore the application of these networks to approximate the solutions of high-dimensional PDEs in the weak sense as future work.

{\bf Acknowledgments.} Y. G. was partially supported by the Ministry of Education in
Singapore under the grant MOE2018-T2-2-147 and MOE AcRF R-146-000-271-112. H. Y. was partially supported by the US
National Science Foundation under award DMS-1945029. C. Z was partially supported by the Ministry of Education in
Singapore under the grant MOE AcRF R-146- 000-271-112 and by NSFC under the grant award 11871364.

\bibliography{expbib}
\bibliographystyle{plain}

\end{document}